\documentclass[a4paper]{article}
\usepackage{amsmath,amssymb,amsthm,amscd,mathrsfs}

\newcommand{\eij}{e_{ij}}

\newcommand{\catc}{\mathcal{C}}
\newcommand{\one}{\mathbf{1}}
\newcommand{\lex}[1]{{ \,\underset{\mathrm{lex}}{#1}\, }}
\newcommand{\rlex}[1]{{ \,\underset{\mathrm{rlex}}{#1}\, }}
\newcommand{\const}{(\text{const.})}
\newcommand{\surj}{\twoheadrightarrow}
\newcommand{\inj}{\hookrightarrow}
\newcommand{\mor}{\rightarrow}
\newcommand{\schub}{\mathfrak{S}}
\newcommand{\smod}{\mathcal{S}}
\newcommand{\der}{\partial}
\newcommand{\borel}{\mathfrak{b}}
\newcommand{\nplus}{\mathfrak{n}^+}
\newcommand{\csa}{\mathfrak{h}}
\newcommand{\ualg}{\mathcal{U}}

\newcommand{\inv}{\mathrm{code}}

\newcommand{\ch}{\mathrm{ch}}
\newcommand{\Ext}{\mathrm{Ext}}
\newcommand{\Hom}{\mathrm{Hom}}
\newcommand{\End}{\mathrm{End}}

\newcommand{\nonneg}{\ZZ_{\geq 0}}
\newcommand{\ZZ}{\mathbb{Z}}

\newcommand{\lmb}{\lambda}
\newcommand{\Lmb}{\Lambda}
\newcommand{\Ker}{\mathrm{Ker}}
\newcommand{\std}{\Delta}
\newcommand{\costd}{\nabla}

\renewcommand{\phi}{\varphi}
\DeclareMathOperator{\supp}{supp}
\DeclareMathOperator{\defect}{def}
\DeclareMathOperator{\hd}{hd}
\DeclareMathOperator{\soc}{soc}
\newtheorem{lem}{Lemma}[section]
\newtheorem{prop}[lem]{Proposition}
\newtheorem{thm}[lem]{Theorem}
\newtheorem{cor}[lem]{Corollary}

\theoremstyle{definition}
\newtheorem{defn}[lem]{Definition}
\newtheorem{rmk}[lem]{Remark}
\newtheorem{eg}[lem]{Example}

\title{Kra\'skiewicz-Pragacz modules and Ringel duality}
\author{
Masaki Watanabe \\ 
Graduate School of Mathematical Sciences, The University of Tokyo, \\
3-8-1 Komaba Meguro-ku Tokyo 153-8914, Japan \\ \texttt{mwata@ms.u-tokyo.ac.jp}}
\date{Last revised \today}

\begin{document}
\maketitle

\noindent\textbf{Abstract.}
In \cite{KP0, KP} Kra\'skiewicz and Pragacz introduced representations of the upper-triangular Lie algebra $\borel$ whose characters are Schubert polynomials. 
In \cite{W} the author studied the properties of Kra\'skiewicz-Pragacz modules using the theory of highest weight categories. 
From the results there, in particular we obtain a certain highest weight category whose standard modules are KP modules. 
In this paper we show that this highest weight category is self Ringel-dual. This leads to an interesting symmetry relation on Ext groups between KP modules. We also show that the tensor product operation on $\borel$-modules is compatible with Ringel duality functor. 

\section{Introduction}
The study of Schubert polynomials is an important and interesting subject in algebraic combinatorics. 
One of the possible methods for studying Schubert polynomials is via certain representations, introduced by Kra\'skiewicz and Pragacz (\cite{KP0, KP}), of the Lie algebra $\borel$ of all upper-triangular matrices. These representations, which we call \textit{KP modules} in this paper, has the property that their character with respect to the subalgebra of diagonal matrices is equal to Schubert polynomials: just like Schur polynomials appear as the characters of irreducible representations of $\mathfrak{gl_n}$. 

In \cite{W} and \cite{W2}, the author investigated the properties of KP modules. The main motivation there was the investigation of Schubert positivity: i.e. the positivity of coefficients in the expansion of a polynomial into a linear combination of Schubert polynomials. 
Because of the property of KP modules described above, if a polynomial $f$ is the character of some $\borel$-module having a filtration by KP modules then $f$ is Schubert-positive. In \cite{W} we gave a characterization of $\borel$-modules having such filtrations, and we used the result in \cite{W2} to prove some positivity results on Schubert polynomials. For details on these results see these two papers. 

One of the main tools used in \cite{W} is the theory of \textit{highest weight categories}. 
A highest weight category is an abelian category with a specified family of objects called \textit{standard objects}, together with some axioms. 
For a highest weight category the notion of \textit{costandard objects} are naturally defined, and it is then shown that an object has a filtration by standard objects (a \textit{standard filtration}) if and only if its extensions with any costandard objects vanish. 
In \cite{W}, with methods from highest weight categories, we obtained a characterization of modules having a filtration by KP modules (the work was strongly inspired by some works by Polo, van der Kallen, Joseph, etc. on Demazure modules: see the bibliography in \cite{W}). 

For a highest weight category $\catc$, there is a notion called \textit{Ringel dual} of $\catc$, denoted $\catc^\vee$. 
It is another highest weight category, with a contravariant equivalence between the subcategories $\catc^\std$ and $(\catc^\vee)^\std$ of objects having standard filtrations: i.e. $\catc^\vee$ is obtained by ``dualizing'' the structure of the standardly-filtered part of $\catc$. 
The main purpose of this paper is to exhibit a highest weight category $\catc_n$ whose standard objects are KP modules (Theorem \ref{kp-hwc}), and show that this highest weight category is self Ringel-dual (Theorem \ref{thm:dual-cn}). 
This provides a contravariant equivalence $\catc_n^\std \mor \catc_n^\std$ on the full subcategory of modules having KP filtrations, and gives an interesting symmetry relation (Corollary \ref{cor:ext}) on the extension groups between KP modules. 
Note that the ordering on the set of weights for our highest weight category
is not the usual root order on the weight lattice (see Section \ref{sec:kp-hw}). 
We also investigate the relation between Ringel duality on $\catc_n$ and the tensor product operation on modules (Theorem \ref{thm:tensor-dual}): we show that the tensor product operation on the $\borel$-modules, slightly modified to give an operation on $\catc_n$ (which is not closed under the usual tensor product), commutes with the Ringel duality functor. 

The paper is organized as follows. In Section \ref{prelimi} we prepare some definitions and results on Schubert polynomials, KP modules, and highest weight categories. 
In Section \ref{sec:kp-hw} we explicitly relate KP modules and highest weight categories: we give a certain highest weight category $\catc_n$ with KP modules $\smod_w$ ($w \in S_n$) being standard modules (it is in fact an immediate consequence of the results in \cite{W}). 
In Section \ref{sec:dual-cn} we show that the Ringel dual of $\catc_n$ is equivalent to itself, and show that under the equivalence $\catc^\std \mor (\catc^\vee)^\std$ the standard modules correspond as $\smod_w \mapsto \smod_{w_0ww_0}$ ($w \in S_n$). 
In Section \ref{sec:tensor-dual} we show the compatibility between tensor product operation and the Ringel duality functor on $\catc_n$. 

\section{Preliminaries}
\label{prelimi}
\subsection{Schubert polynomials}
A \textit{permutation} is a bijection from the set of all positive integers to itself which fixes all but finitely many points. 
Throughout this paper let us fix a positive integer $n$. 
Let $S_n =\{w: \text{permutation, $w(i)=i$ ($i > n$)} \}$
and $S_\infty^{(n)} = \{w : \text{permutation, $w(n+1)<w(n+2)<\cdots$}\}$. 
For $i < j$, let $t_{ij}$ denote the permutation which exchanges $i$ and $j$ and fixes all other points. 
Let $s_i=t_{i,i+1}$. 
For a permutation $w$, let $\ell(w)=\#\{i<j : w(i)>w(j)\}$. 
Let $w_0 \in S_n$ be the longest element of $S_n$, i.e. $w_0(i)=n+1-i$ ($1 \leq i \leq n$). 
For $w \in S_\infty^{(n)}$ we define $\inv(w)=(\inv(w)_1, \ldots, \inv(w)_n) \in \nonneg^n$ by $\inv(w)_i=\#\{j : i<j, w(i)>w(j)\}$: this is usually called the \textit{Lehmer code} and it uniquely determines $w$. 
 Note that if $w \in S_n$ we have $\inv(w) \in \Lmb_n := 
\{(a_1, \ldots, a_n) \in \ZZ^n : 0 \leq a_i \leq n-i \}$. 

For a polynomial $f=f(x_1, \ldots, x_n)$ and $1 \leq i \leq n-1$, we define $\der_if=\frac{f-s_if}{x_i-x_{i+1}}$. 
For each $w \in S_\infty^{(n)}$ we can assign its \textit{Schubert polynomial} $\schub_w \in \ZZ[x_1, \ldots, x_n]$, which is recursively defined by 
\begin{itemize}
\item $\schub_{w}=x_1^{w(1)-1}x_2^{w(2)-1} \cdots x_n^{w(n)-1}$ if $w(1)>\cdots>w(n)$, and
\item $\schub_{ws_i}=\der_i\schub_w$ if $\ell(ws_i)<\ell(w)$. 
\end{itemize}
It is known that the set $\{\schub_w : w \in S_n\}$
constitutes a $\ZZ$-linear basis of the ring $H_n=\ZZ[x_1, \ldots, x_n]/I$ where $I$ is the ideal generated by the homogeneous symmetric polynomials in $x_1, \ldots, x_n$ of positive degrees. 
Note also that the natural map $\ZZ[x_1, \ldots, x_n] \surj H_n$ restricts to an isomorphism of $\ZZ$-modules $\bigoplus_{\lmb \in \Lmb_n} \ZZ x^\lmb \overset{\sim}\mor H_n$, where $x^\lmb=x_1^{\lmb_1} \cdots x_n^{\lmb_n}$ (\cite[Proposition 2.5.3, Corollary 2.5.6]{Man}). 
We also need the following basic facts:
\begin{prop}
\label{prop:hn-invol}
Let $\iota: H_n \mor H_n$ be the ring automorphism given by $\overline{x_i} \mapsto -\overline{x_{n+1-i}}$ where $\overline{x_i} = x_i \bmod I$. Then for $w \in S_n$, $\iota(\schub_w) = \schub_{w_0ww_0}$. 
\end{prop}
\begin{rmk}
The automorphism $\iota$ corresponds to the map between flag varieties which takes a flag to its dual flag: see eg. \cite[\S 10.6, Exercise 13]{Ful}
\end{rmk}
\begin{proof}
First note that $\iota \circ \der_i \circ \iota = \der_{n-i}$. Thus we only have to check the proposition for $w=w_0$. 

Since the only elements in $H_n = \bigoplus_{w \in S_n} \ZZ \schub_w$ with degree $\binom{n}{2}$ are the constant multiples of $\schub_{w_0}$, we see that $\iota(\schub_{w_0})$ is a constant multiple of $\schub_{w_0}$. 
Let $(i_1, \ldots, i_l)$ be a longest word, i.e. $l=\ell(w_0)$ and $w=s_{i_1} \cdots s_{i_l}$. Note that $(n-i_1, \ldots, n-i_l)$ is also a longest word. 
We have $\der_{i_1}\cdots \der_{i_l} \schub_{w_0} = \schub_{\mathrm{id}} = 1$ and $\der_{i_1} \cdots \der_{i_l} \iota(\schub_{w_0}) = (\iota\der_{n-i_1}\iota) \cdots (\iota\der_{n-i_l}\iota) \iota(\schub_{w_0}) = \iota(\der_{n-i_1}\cdots \der_{n-i_l} \schub_{w_0}) = 1$. Thus $\iota(\schub_{w_0})=\schub_{w_0}$. 
\end{proof}

\begin{prop}
\label{prop:schub-vanish}
For $w \in S_\infty^{(n)} \smallsetminus S_n$ we have $\schub_w \in I$.
\end{prop}
\begin{proof}
Since $\der_i I \subset I$ for any $1 \leq i \leq n-1$, it suffices to show that the proposition holds in the case $w(1)>\cdots>w(n)$. 
Since in this case $\schub_w = x_1^{w(1)-1}x_2^{w(2)-1} \cdots x_n^{w(n)-1}$ it is enough to show $x_1^n \in I$. 
This is immediate from the equation $\prod_{i=2}^n (1-\overline{x_i}u) = \frac{1}{1-\overline{x_1}u} = \sum_{j \geq 0} \overline{x_1}^ju^j$ in $H_n [[ u ]]$ since the LHS has no terms of degrees $\geq n$ in $u$. 
\end{proof}

Schubert polynomials satisfy the following Cauchy identity: 
\begin{prop}[{{\cite[$(5.10)$]{Mac}}}]
$\sum_{w \in S_n} \schub_w(x) \schub_{ww_0}(y) = \prod_{i+j \leq n} (x_i+y_j)$. 
\end{prop}

\subsection{Kra\'skiewicz-Pragacz modules}
Let $K$ be a field of characteristic zero. 
Let $\borel=\borel_n$ be the Lie algebra of all upper triangular $K$-matrices. 
and let $\csa \subset \borel$ and $\nplus \subset \borel$ be the subalgebra of all diagonal matrices
and the subalgebra of all strictly upper triangular matrices respectively. 
Let $\ualg(\borel)$ and $\ualg(\nplus)$ be the universal enveloping algebras of $\borel$ and $\nplus$ respectively. 
For a $\ualg(\borel)$-module $M$ and $\lmb = (\lmb_1, \ldots, \lmb_n) \in \ZZ^n$, 
let $M_\lmb = \{m \in M : hm=\langle \lmb,h \rangle m \;\text{($\forall h \in \csa$)}\}$ where $\langle \lmb,h \rangle = \sum \lmb_i h_i$. 
$M_\lmb$ is called the \textit{$\lmb$-weight space} of $M$. 
If $M_\lmb \neq 0$ then $\lmb$ is said to be a \textit{weight} of $M$. 
If $M=\bigoplus_{\lmb \in \ZZ^n} M_\lmb$ and each $M_\lmb$ has finite dimension, 
then we call that $M$ is a \textit{weight $\borel$-module}
and define $\ch(M)=\sum_{\lmb} \dim M_\lmb x^\lmb$. 
From here we only consider weight $\borel$-modules. 
For $1 \leq i \leq j \leq n$, let $\eij \in \borel$ be the matrix with $1$ at the $(i,j)$-position and all other coordinates $0$. 
Let $\rho=(n-1, n-2, \ldots, 0) \in \ZZ^n$ and $\one = (1, \ldots, 1) \in \ZZ^n$. 
Also let $\alpha_{ij}=(0, \ldots, 0, 1, 0, \ldots, 0, -1, 0, \ldots, 0) \in \ZZ^n$ for $1 \leq i < j \leq n$, where $1$ and $-1$ are at the $i$-th and $j$-th positions respectively. 

For $\lmb \in \ZZ^n$, let $K_\lmb$ denote the one-dimensional $\ualg(\borel)$-module where $h \in \csa$ acts by $\langle \lmb,h \rangle$ and $\nplus$ acts by $0$. 

In \cite{KP}, Kra\'skiewicz and Pragacz defined certain $\ualg(\borel)$-modules which here we call \textit{Kra\'skiewicz-Pragacz modules} or \textit{KP modules}. 
Here we use the following definition. 
Let $w \in S_\infty^{(n)}$. 
Let $K^n=\bigoplus_{1 \leq i \leq n} K u_i$ be the vector representation of $\borel$. 
For each $j \geq 1$, let $\{i < j : w(i)>w(j)\}=\{i_{j,1}, \ldots, i_{j,l_j}\}$ ($i_{j,1}<\cdots<i_{j,l_j}$), 
and let $u_w^{(j)}=u_{i_{j,1}} \wedge \cdots \wedge u_{i_{j,l_j}} \in \bigwedge^{l_j} K^n$. 
Let $u_w=u_w^{(1)} \otimes u_w^{(2)} \otimes \cdots \in \bigotimes_{j \geq 1} \bigwedge^{l_j} K^n$. 
Then the KP module $\smod_w$ associated to $w$ is defined as $\smod_w=\ualg(\borel)u_w=\ualg(\nplus)u_w$. 

\begin{eg}
If $w=s_i$, then $u_w=u_i$ and thus we have $\smod_w=K^i := \bigoplus_{1 \leq j \leq i} Ku_j$. 
\end{eg}

KP modules have the following property: 
\begin{thm}[{{\cite[Remark 1.6, Theorem 4.1]{KP}}}]
For any $w \in S_\infty^{(n)}$, $\smod_w$ is a weight $\borel$-module and $\ch(\smod_w)=\schub_w$. 
\end{thm}

We slightly generalize the notion of KP modules. 
For $\lmb \in \ZZ^n$, take $k \in \ZZ$ so that $\lmb+k\one \in \nonneg^n$ and define $\smod_{\lmb} = \smod_w \otimes K_{-k\one}$ ($w \in S_\infty^{(n)}$, $\inv(w)=\lmb+k\one$). Note that this definition does not depend on the choice of $k$. 
We denote by $u_\lmb$ the generator of $\smod_\lmb$. 
We also write $\schub_\mu = \ch(\smod_\mu)$ for $\mu \in \ZZ^n$. 

\subsection{Highest weight categories}
Highest weight categories were first introduced by Cline, Parshall and Scott (\cite{CPS}). In this paper we use the following definition (cf. \cite{Los}):
\begin{defn}
Let $\catc$ be an abelian $K$-category with enough projectives and injectives, such that every object has finite length. 
Let $\Lmb = (\Lmb, \leq)$ be a finite poset indexing the simple objects $\{L(\lmb) : \lmb \in \Lmb \}$ in $\catc$ (called the \textit{weight poset}). Moreover, assume that a family of objects $\{\std(\lmb) : \lmb \in \Lmb\}$ called \textit{standard objects} is given.  
Then $\catc = (\catc, \Lmb, \{\std(\lmb)\})$ is called \textit{a highest weight category}
if the following axioms hold:
\begin{enumerate}
\item $\Hom_\catc(\std(\lmb), \std(\mu))=0$ unless $\lmb \leq \mu$. 
\item $\End_\catc(\std(\lmb)) \cong K$. 
\item Let $P(\lmb)$ denote the projective cover of $L(\lmb)$. Then there exists a surjection $P(\lmb) \surj \std(\lmb)$ such that its kernel admits a filtration whose successive quotients are of the form $\std(\nu)$ ($\nu > \lmb$). 
\end{enumerate}
\label{defn-hwc}
\end{defn}

Below we list some properties of highest weight categories which are used in this paper. For the proofs of these see Appendix (or references such as \cite{CPS}, \cite{Rin} and \cite[Appendix]{Don}: the formulations of highest weight categories and these properties vary with references, and so we collect both the definitions and its basic properties we use, along with their proofs, in the appendix in order to adapt them into our settings). 

For a highest weight category $\catc$ and $\lmb \in \Lmb$, let $\catc_{\leq \lmb}$ denote the full subcategory of objects whose simple constituents are all of the form $L(\mu)$ ($\mu \leq \lmb$). Let $\costd(\lmb) \in \catc$ be the injective hull of $L(\lmb)$ in the subcategory $\catc_{\leq \lmb}$. The objects $\costd(\lmb)$ are called \textit{costandard objects}. The standard modules can also be characterized in this way: $\std(\lmb)$ is the projective cover of $L(\lmb)$ in $\catc_{\leq \lmb}$. 
More generally, for any order ideal $\Lmb'$ containing $\lmb$ as one of its maximal elements, $\std(\lmb)$ is the projective cover of $L(\lmb)$ in the full subcategory $\catc_{\Lmb'}$ of modules whose simple constituents are $L(\mu)$ with $\mu \in \Lmb'$ (Proposition \ref{std-projectivity}). 

A \emph{standard} (resp.\ \emph{costandard}) filtration of an object $M \in \catc$ is a filtration such that each of its successive quotients is isomorphic to some standard (resp.\ costandard) object. For a highest weight category $\catc$ let $\catc^\std$ denote the subcategory of all objects having standard filtrations. 

\begin{prop}[Proposition \ref{std-costd-ext}]
\label{prop:hw-homext}
$\Hom_\catc(\std(\lmb), \costd(\mu)) \cong K$ if $\lmb=\mu$ and $0$ otherwise, 
and $\Ext^i_\catc(\std(\lmb), \costd(\mu))=0$ for $i \geq 1$. 
Hence if $M \in \catc$ has a standard (resp.\ costandard) filtration, 
then for any $\lmb \in \Lmb$, the number of times $\std(\lmb)$ (resp.\ $\costd(\lmb)$)
appears in (any) standard (resp.\ costandard) filtration is $\dim \Hom_\catc(M, \costd(\lmb))$ (resp.\ $\dim \Hom_\catc(\std(\lmb), M)$). 
\end{prop}

\begin{prop}[Proposition \ref{filtr-lem}]
\label{prop:hw-quot}
Let $M \in \catc^\std$ and let $\Lmb' \subset \Lmb$ be an order ideal.
Let $M^{\Lmb'}$ be the largest quotient of $M$
each of whose simple constituent is isomorphic to some $L(\lmb)$ ($\lmb \in \Lmb'$). Note that the existence of such a quotient follows from the finiteness of the length of $M$. 
Then $M^{\Lmb'}$ and $\Ker(M \surj M^{\Lmb'})$ have filtrations whose subquotients are of the form $\std(\lmb)$ ($\lmb \in \Lmb'$) and $\std(\mu)$ ($\mu \in \Lmb \smallsetminus \Lmb'$) respectively. 
If $\lmb \in \Lmb'$ is a maximal element in $\Lmb'$, then $\Ker(M^{\Lmb'} \surj M^{\Lmb' \smallsetminus \{\lmb\}})$ is isomorphic to a direct sum of some copies of $\std(\lmb)$ (note that the former statement follows from this one). 
\end{prop}

An object $M \in \catc$ is called \textit{a tilting} or \textit{a tilting object} if $M \in \catc^\std$ and $\Ext^1(\std(\lmb), M) = 0$ for all $\lmb \in \Lmb$. 
In \cite{Rin} Ringel showed the following results: 

\begin{prop}[Proposition \ref{indec-tilt}]
For each $\lmb$, there exists a unique (up to isomorphism) tilting $T(\lmb)$ which is indecomposable and has the property that there exists an injection $\std(\lmb) \inj T(\lmb)$ whose cokernel admits a filtration by the objects of the form $\std(\mu)$ ($\mu < \lmb$). 
Moreover, every tilting is a direct sum of the objects $T(\lmb)$ ($\lmb \in \Lmb$). 
\end{prop}

\begin{prop}[Proposition \ref{ringel-ext}, \ref{ringel-hwc}, \ref{ringel-dual}]
\label{prop:generaldual}
Let $T$ be a tilting which contains every $T(\lmb)$ at least once as its direct summand (such $T$ is called a \emph{full tilting}). 
Then the category $\catc^\vee$ of all finite dimensional left $End_{\catc}(T)$-modules (which in fact does not depend, up to equivalence, on a choice of $T$), called the \emph{Ringel dual} of $\catc$, is again a highest weight category
with standard objects $\std^\vee(\lmb) = \Hom_\catc(\std(\lmb), T)$ and the weight poset $\Lmb^{\mathrm{op}}$, the opposite poset of $\Lmb$. 
Moreover, the contravariant functor $F: \catc \mor \catc^\vee$ given by $FM=\Hom_\catc(M,T)$ restricts to a contravariant equivalence between $\catc^\std$ and $(\catc^\vee)^\std$, and gives an isomorphism $\Ext^i_\catc(M,N) \cong \Ext^i_{\catc^\vee}(FN,FM)$ for any $M, N \in \catc^\std$ and any $i \geq 0$. 
\end{prop}

\section{KP modules and highest weight categories}
In this section we introduce certain highest weight categories whose standard modules are KP modules, using the results from \cite{W}. As we noted in the introduction, the ordering of the weights used here is different from the usual root order on the weight lattices. 
\label{sec:kp-hw}

Let us introduce two ordering relations on $\nonneg^n$ as follows. 
For $\lmb=\inv(w)$ and $\mu=\inv(v)$ ($\lmb, \mu \in \nonneg^n, w, v \in S_\infty^{(n)}$) with $\ell(w)=\ell(v)$, 
we define $\lmb < \mu \iff w^{-1} \lex> v^{-1}$ and $\lmb <' \mu \iff w^{-1} \rlex> v^{-1}$ (if $\ell(w) \neq \ell(v)$ we define $\lmb$ and $\mu$ to be incomparable). 
Here for two permutations $x$ and $y$, 
$x \lex> y$ (resp.\ $x \rlex> y$) if there exists an $i$ such that $x(j) = y(j)$ for any $j < i$ (resp.\ $j > i$) and $x(i)>y(i)$. 
We write $\lmb \prec \mu$ if both $\lmb < \mu$ and $\lmb <' \mu$ hold. 
For general $\lmb, \mu \in \ZZ^n$, take $k \in \ZZ$ so that $\lmb+k\one, \mu+k\one \in \nonneg^n$, and write $\lmb < \mu$ (resp.\ $\lmb <' \mu$, $\lmb \prec \mu$) iff $\lmb+k\one < \mu+k\one$ (resp.\ $\lmb+k\one <' \mu+k\one$, $\lmb+k\one \prec \mu+k\one$). This definition does not depend on $k$. 
As we have shown in \cite[Lemma 6.2]{W}, $\lmb < \mu$ if and only if $\rho-\lmb <' \rho-\mu$. 

It can be seen that $\Lmb_n$ is an order ideal of $\ZZ^n$ with respect to $\prec$, using \cite[Lemma 6.2]{W} and the argument in the proof of \cite[Lemma 6.3]{W}. 
\begin{thm}
Let $\Lmb \subset \ZZ^n$ be a finite order ideal with respect to the ordering $\prec$. 
Let $\catc_\Lmb$ be the category of all weight $\borel$-modules 
whose weights are in $\Lmb$ (note that this notation agrees with the notation in the theory of highest weight categories, since the simple objects in the category of weight  $\borel$-modules are just the one-dimensional modules $K_\mu$). 
Then $\catc_\Lmb$ is a highest weight category with weight poset $(\Lmb, \prec)$
and standard objects $\{\smod_\lmb : \lmb \in \Lmb \}$. 
In particular, $\catc_n := \catc_{\Lmb_n}$ is a highest weight category. 
\label{kp-hwc}
\end{thm}
\begin{proof}
In \cite[Proposition 6.4]{W} we showed that if $\lmb, \mu \in \ZZ^n$ and $(\smod_\mu)_\lmb \neq 0$ then $\lmb \preceq \mu$
(more precisely, $\lmb \leq \mu$ follows from (1) in the proof of \cite[Proposition 6.4]{W}, and $\lmb \leq' \mu$ follows from (2) in the same proof and \cite[Lemma 6.2]{W}). 
If $\Hom(\smod_\lmb, \smod_\mu) \neq 0$, then $(\smod_\mu)_\lmb \neq 0$ must hold since $\smod_\lmb$ is generated by an element of weight $\lmb$, and thus $\lmb \preceq \mu$. This verifies (1) in the definition of highest weight category. (2) also follows since $(\smod_\lmb)_\lmb = Ku_\lmb$. 

Let us verify (3). 
Let the elements of $\Lmb$ be indexed as $\lmb^1, \ldots, \lmb^l$ so that $\lmb^i \prec \lmb^j$ implies $i<j$.  Then $\{\lmb^1, \ldots, \lmb^i\}$ is an order ideal. 
Let $\lmb = \lmb^k \in \Lmb$. 
Let $P_\lmb$ be the projective cover of $K_\lmb$ in the category of all weight $\borel$-modules, so $P_\lmb \cong \ualg(\borel)/\langle h-\langle \lmb, h \rangle \rangle_{h \in \csa}$. 

Let $P^i$ denote the largest quotient of $P_\lmb$ such that all of its weights are in $\{ \lmb^1, \ldots, \lmb^i \}$. 
Note that $P^i$ is the projective cover of $K_\lmb$ in $\catc_{\{\lmb^1, \ldots, \lmb^i\}}$ for $i \geq k$. 
In particular, $P^l$ is the projective cover of $K_\lmb$ in $\catc_\Lmb$, and $P^k \cong \smod_\lmb$ as we showed in \cite[Proposition 6.4]{W}. 
The same argument as in the proof of \cite[Lemma 7.1]{W} shows that
the kernel $\Ker(P^i \surj P^{i-1})$ of the natural surjection is isomorphic to a direct sum of some copies of $\smod_{\lmb^i}$. 
Thus 
\[0 \subset \Ker(P^l \surj P^{l-1}) \subset \Ker(P^l \surj P^{l-2}) \subset \cdots \subset \Ker(P^l \surj P^k) \subset P^l
\] is a KP filtration, and the last successive quotient is isomorphic to $P^k \cong \smod_\lmb$. This shows (3). 
\end{proof}

From \cite[Proposition 6.4]{W} (and \cite[Lemma 6.2]{W}) we see that the costandard modules in $\catc_\Lmb$ are $\{\smod_{\rho-\lmb}^* \otimes K_\rho : \lmb \in \Lmb\}$. 

\begin{rmk}
It is easy to see that the projective cover of $K_\lmb$ in $\catc_\Lmb$ ($\lmb \in \Lmb \subset \ZZ^n$) is given by the largest quotient $(P_\lmb)^{\Lmb}$ of $P_\lmb$ each of whose simple constituents are $L(\mu)$ ($\mu \in \Lmb$). 
Thus from the theorem above we see that $\smod_\lmb \cong (P_\lmb)^{\Lmb}$ for any order ideal $\Lmb \subset (\ZZ^n, \prec)$ containing $\lmb$ as one of its maximal elements. 

Let $\Lmb \subset (\ZZ^n, \prec)$ be an order ideal and $\lmb \in \Lmb$ be one of its maximal element as above. If a weight $\borel$-module $M$ is generated by an element of weight $\lmb$ then $M$ is a quotient of $P_\lmb$. So if in addition $M \in \catc_{\Lmb}$ then it follows that $M$ is in fact a quotient of $\smod_\lmb$. 
\label{rmk-quot-s}
\end{rmk}

\section{Ringel dual of $\catc_n$}
\label{sec:dual-cn}
In this section we show the following: 

\begin{thm}
\label{thm:dual-cn}
The Ringel dual of the highest weight category $\catc_n$ is equivalent to $\catc_n$ itself. 
The functor $F$ in Proposition \ref{prop:generaldual} acts on the standard modules by $F(\smod_w) = \smod_{w_0ww_0}$. 
\end{thm}

From this theorem in particular we obtain the following symmetry relation for the Hom and Ext groups between KP modules: 
\begin{cor}
\label{cor:ext}
$\Ext^i_{\catc_n}(\smod_w, \smod_v) \cong \Ext^i_{\catc_n}(\smod_{w_0vw_0}, \smod_{w_0ww_0})$ for any $w, v \in S_n$ and any $i \geq 0$. 
\hfill $\Box$
\end{cor}

\begin{rmk}
Since $\Lmb_n$ is an order ideal in $(\ZZ^n, \prec)$, by the same argument as in \cite[Lemma 7.1]{W} it holds that $\Ext^i_{\catc_n}(M, N) \cong \Ext^i(M, N)$ (Ext group in the category of all weight $\borel$-modules) for any $M, N \in \catc_n$. Hence Corollary \ref{cor:ext} in fact shows $\Ext^i(\smod_w, \smod_v) \cong \Ext^i(\smod_{w_0vw_0}, \smod_{w_0ww_0})$. 
\end{rmk}

\begin{rmk}
By Theorem \ref{thm:dual-cn}, 
we have a functor defined as the composition of
$\catc_n \overset{F}{\mor} \catc_n \inj \catc_{n+1} \overset{F}{\mor} \catc_{n+1}$. 
By the theorem we see that this functor acts on the standard modules by $\smod_w \mapsto \smod_{1 \times w}$, where $1 \times w \in S_{n+1}$ is defined by $(1 \times w)(1)=1$, $(1 \times w)(i+1)=w(i)+1$ ($1
\leq i \leq n$). 
By the corollary and the remark above we see that the Ext groups between KP modules are stable under shifting, i.e. $\Ext^i(\schub_w, \schub_v) \cong \Ext^i(\schub_{1 \times w}, \schub_{1 \times v})$. 
\end{rmk}

Let us move to the proof of Theorem \ref{thm:dual-cn}.
First we prepare some definitions and results. 
For $\lmb=\inv(w) \in \Lmb_n$ define $\overline{\lmb}=\inv(w_0ww_0)$. 
Note that by definition, for $\lmb, \mu \in \Lmb_n$, $\lmb \leq \mu$ iff $\overline{\lmb} \geq' \overline{\mu}$. 
For each $\lambda \in \Lmb_n$, define $T(\lmb)=\bigotimes_{1 \leq j \leq n-1} \bigwedge^{\overline{\lmb}_j} K^{n-j}$. 

As we showed in the proof of \cite[Lemma 4.2]{W2}, $T(\lmb)$ has a filtration whose subquotients are of the form $\smod_\mu$ ($\mu \in \Lmb_n, \mu \preceq \lmb$). 
Since $\overline{\rho-\lmb} = \rho-\overline{\lmb}$ we have $T(\lmb) \cong T(\rho-\lmb)^* \otimes K_\rho$, and thus $T(\lmb)$ also has a filtration whose subquotients are of the form $\smod_{\rho-\nu}^* \otimes K_\rho$ ($\nu \in \Lmb_n, \nu \preceq \lmb$). 
Thus by Proposition \ref{prop:hw-homext} we see that $\Ext^1(\smod_\mu, T(\lmb))=0$ for all $\mu \in \Lmb_n$. 
Thus $T(\lmb)$ is a tilting in $\catc_n$. 

Since the weights of $\smod_\mu$ are all $\preceq \mu$ (\cite[Proposition 6.4]{W}), the weights of $T(\lmb)$ are all $\preceq \lmb$ and the weight space $T(\lmb)_{\lmb}$ is one-dimensional. 
By these properties we see that $T(\lmb)$ contains the indecomposable tilting module corresponding to $\lmb$ (in fact, we will see that $T(\lmb)$ is an indecomposable tilting). So if we define $T=\bigoplus_{\lmb \in \Lmb_n} T(\lmb) \cong \bigwedge^\bullet(K^{n-1} \oplus K^{n-2} \oplus \cdots \oplus K^1)$, then $T$ is a full tilting. 

By Proposition \ref{prop:hw-quot}, if $M$ has a standard filtration, 
then $\Ker(M^{\leq' \lmb} \surj M^{<' \lmb})$ is isomorphic to a direct sum of copies of $\smod_\lmb$, where $M^{\leq' \lmb}$ and $M^{<' \lmb}$ are the largest quotients of $M$ whose weights are all $\leq' \lmb$ and $<' \lmb$ respectively. 
In this case we see, by the same argument as in the proof of \cite[Theorem 8.1]{W} (using the ordering $\leq'$ instead of $\leq$),  
that the isomorphism can be written as $\smod_\lmb \otimes (M^{\leq' \lmb})_\lmb \ni xu_\lmb \otimes v \mapsto xv \in \Ker(M^{\leq' \lmb} \surj M^{<' \lmb})$, where on the left-hand side $\borel$ acts only on $\smod_\lmb$. 

\begin{proof}[Proof of Theorem \ref{thm:dual-cn}]
Let $\catc=\catc_n$. 
Throughout this proof and thereafter
we write $\Hom$, $\End$ and $\Ext^i$ for $\Hom_\catc$, $\End_\catc$ and $\Ext^i_\catc$ respectively. 

Let $\borel' = \bigoplus_{i \leq j} K e_{ij}'$ be a copy of $\borel$. 
We define an action of $\borel'$ on $T$ which commutes with the action of $\borel$ as follows. 
Take a basis $\{u_{ij} : i,j \geq 1, i+j \leq n\}$ of $K^{n-1} \oplus \cdots \oplus K^1$
so that the action of $\borel$ is given by $e_{pq}u_{ij}=\delta_{qi}u_{pj}$. 
Then we define the action of $\borel'$ on $K^{n-1} \oplus \cdots \oplus K^1$ by $e_{pq}'u_{ij} = \delta_{qj}u_{ip}$, and define the action on $T$ as the one induced from this action. 
In other words, if $-': T \mor T$ is the involution given by $u_{ij}' = u_{ji}$, 
then $e_{ij}' =  -' \circ e_{ij} \circ -'$. 
Since the actions of $\borel$ and $\borel'$ commute, we have an algebra homomorphism $\ualg(\borel) \cong \ualg(\borel') \mor \End_\borel(T)$, 
and thus an $\End_\borel(T)$-module can be naturally seen as a $\ualg(\borel)$-module (beware that, as we have remarked before, we will simply write $\End(T)$ to mean $\End_\borel(T) = \End_\catc(T)$ hereafter). 
If $M$ is an $\End(T)$-module, 
then its weight-space decomposition as a $\ualg(\borel)$-module is given by
$M = \bigoplus_{\lmb \in \Lmb_n} M_\lmb = \bigoplus_{\lmb \in \Lmb_n} \pi_\lmb M$,
where $\pi_\lmb \in \End(T)$ is the projection $T=\bigoplus_{\mu \in \Lmb_n} T(\mu) \surj T(\overline{\lmb})$;
 in particular the weights of $M$ are all in $\Lmb_n$. So we have a functor $\catc^{\vee} = \text{$\End(T)$-mod} \mor \catc$. We want to show that this functor is an equivalence and the composition $\catc \overset{F = \Hom(-,T)}{\longrightarrow} \catc^\vee \overset{\sim}{\mor} \catc$ sends $\smod_w$ to $\smod_{w_0ww_0}$. 

First we show the second claim. 
By definition, $\smod_w$ is isomorphic to the $\borel$-submodule of $T$
generated by $\bigwedge_{i<j, w(i)>w(j)} u_{i,n+1-j}$; hereafter we identify $u_w$ with this element. 
Note that $u_w' = \pm u_{w_0ww_0}$.
We have an injective homomorphism
$\smod_{w_0ww_0} \mor \Hom(\smod_w, T)$ given by $xu_{w_0ww_0} \mapsto (v \mapsto x'v)$ ($x \in \ualg(\borel)$): 
it is well-defined since $x u_{w_0ww_0}=0$ implies $x'yu_w = y(xu_{w_0ww_0})'=0$ for any $y \in \ualg(\borel)$, 
and it is injective because $v \mapsto x'v$ maps $u_w = \pm u_{w_0ww_0}'$ to $\pm(xu_{w_0ww_0})'$. 
Since $T$ has a costandard filtration, by Proposition \ref{prop:hw-homext} the dimension of $\Hom(\smod_w, T)$ is equal to the number of times the costandard module $\smod_{w_0w}^* \otimes K_\rho$ appears in (any) costandard filtration of $T$. 
Since $T \cong T^* \otimes K_\rho$, this number is equal to the number of times $\smod_{w_0w}$ appears in (any) standard filtration of $T$. 
From Cauchy identity we see that $\ch(T)=\prod (x_i+1)^{n-i} = \sum_{v \in S_n} \schub_v(x) \schub_{vw_0}(1)$, and thus we see that $\dim \Hom(\smod_w, T) = \schub_{w_0ww_0}(1) = \dim \smod_{w_0ww_0}$. So the injection above is in fact an isomorphism and this shows the second claim. 

Now let us show that the functor $\catc^\vee \mor \catc$ given above is an equivalence. 
First we note the following thing. 
Define an algebra $A=\ualg(\borel)/I$, where $I$ is the two-sided ideal generated by
all elements in $\ualg(\csa) = S(\csa) \cong K[\csa^*]$ which vanish on $\Lmb_n$ (here $\Lmb_n \subset \ZZ^n$ is identified with a subset of $\csa^*$ via the pairing $\langle \lmb, h \rangle = \sum_i \lmb_i h_i$ introduced before). 
Then the objects in $\catc$, i.e. weight $\borel$-modules with weights in $\Lmb_n$, are just the finite dimensional $A$-modules (note that $A$-modules automatically have weight decompositions since any element $p_\lmb \in K[\csa^*]$ such that $p_\lmb(\mu)=\delta_{\lmb\mu}$ ($\forall \mu \in \Lmb_n$) acts as a projection onto the $\lmb$-weight space). 
Thus it suffices to show that the map
\[
\phi : A \ni a \mapsto (\text{$\borel'$-action of $a$ on $T$}) \in \End(T)
\]
is an isomorphism. 
We note here that $A$ has an algebra anti-automorphism $\iota$
defined by $\iota(h)=\langle \rho, h \rangle-h$ ($h \in \csa$) 
and $\iota(e_{ij}) = -e_{ij}$ ($1 \leq i < j \leq n$). 
For each $\lmb \in \Lmb_n$ take $p_\lmb \in A$ as above. 
Note that $\iota(p_\lmb)=p_{\rho-\lmb}$.

Let $0 \leq d \leq \binom{n}{2}$. 
It suffices to show that $\phi$ induces an isomorphism between $A_d := \sum_{\lmb_1+\cdots+\lmb_n=d} A p_\lmb A$ and $\End(T)_d := \End(\bigwedge^d (K^{n-1} \oplus \cdots \oplus K^1))$, 
since as algebras $A = \bigoplus_d A_d$ (this follows easily from $h p_\lmb=p_\lmb h$ and $e_{ij} p_\lmb=p_{\lmb-\alpha_{ij}}e_{ij}$) and $\End(T)=\bigoplus_d \End(T)_d$. 
So let us fix such $d$ hereafter in this proof. 
Let the elements of $\{\lmb \in \Lmb_n : \sum \lmb_i = d\}$ be $\lmb^{(1)} > \lmb^{(2)} > \cdots > \lmb^{(r)}$. Note $\overline{\lmb^{(1)}} <' \overline{\lmb^{(2)}} <' \cdots <' \overline{\lmb^{(r)}}$. 
Define $I_k = \sum_{\mu \geq \lmb^{(k)}} A p_\mu A$. 
Also define $J_k = \Hom(T^{\leq' \overline{\lmb^{(k)}}}, T) \subset \End(T)$ where $T^{\leq' \overline{\lmb^{(k)}}}$ is the largest quotient of $T$ whose weights are all $\leq' \overline{\lmb^{(k)}}$. 
In other words, $J_k$ consists of all morphisms in $\End(T)$ which vanishes on the weight spaces $T_\mu$ ($\mu \not\leq' \overline{\lmb^{(k)}}$).
Define $I_0=0$ and $J_0=0$. 
Note that $I_0 \subset \cdots \subset I_r=A_d$ and $J_0 \subset \cdots \subset J_r=\End(T)_d$. 
It suffices to show that $\phi(I_k) \subset J_k$ and that $\phi$ induces an isomorphism $I_k/I_{k-1} \mor J_k/J_{k-1}$ for all $1 \leq k \leq r$. 

Fix $1 \leq k \leq r$. Let $\lmb=\lmb^{(k)}$. 
The first claim $\phi(I_k) \subset J_k$ follows since for $\mu \geq \lmb$, $p_\mu$ acts on $T$ as the projection onto $T(\overline{\mu})$, and every weight $\nu$ of $T(\overline{\mu})$ satisfies $\nu \leq' \overline{\mu} \leq' \overline{\lmb}$. 
Let us now show that the induced map $I_k/I_{k-1} \mor J_k/J_{k-1}$ is an isomorphism. We show that $I_k/I_{k-1}$ and $J_k/J_{k-1}$ are both isomorphic to $\smod_{\lmb} \otimes \smod_{\rho-\lmb}$ as vector spaces and that the composition of isomorphisms $I_k/I_{k-1} \cong \smod_{\lmb} \otimes \smod_{\rho-\lmb} \cong J_k/J_{k-1}$ coincides with the map induced from $\phi$. 

We first show that $I_k / I_{k-1} \cong \smod_\lmb \otimes \smod_{\rho-\lmb}$. 
First note that $A$ is a projective object in $\catc = \text{$A$-mod}$. Since projective objects in $\catc$ have standard filtrations, $A$ has a standard filtration. 

By definition, $A/I_k \cong A^{< \lmb}$ and $A/I_{k-1} \cong A^{\leq \lmb}$, 
and thus $I_k/I_{k-1} \cong \Ker(A^{\leq \lmb} \mor A^{< \lmb})$. 
By Proposition \ref{prop:hw-quot} this is a direct sum of $m$ copies of $\smod_\lmb$, 
where $m$ is the number of times $\smod_\lmb$ appears in a standard filtration of $A$. 
This number $m$ can be calculated, by Proposition \ref{prop:hw-homext}, as $\dim \Hom(A, \smod_{\rho-\lmb}^* \otimes K_\rho) = \dim (\smod_{\rho-\lmb}^* \otimes K_\rho) = \dim \smod_{\rho-\lmb}$. 
Thus $\smod_\lmb \otimes \smod_{\rho-\lmb}$ and $I_k/I_{k-1}$ have the same dimensions. 
We claim that the map $\smod_\lmb \otimes \smod_{\rho-\lmb} \ni xu_\lmb \otimes yu_{\rho-\lmb} \mapsto x p_\lmb \iota(y) = \iota(y \iota(x p_\lmb)) \in I_k/I_{k-1}$ is well-defined. 
To see this, first observe that the weights of $I_k / I_{k-1}$ (resp.\ $\iota(I_k / I_{k-1})$ are all $\leq \lmb$ (resp.\ $\leq' \rho-\lmb$).
Thus the submodule of $I_k / I_{k-1}$ (resp.\ $\iota(I_k / I_{k-1})$) generated by $p_\lmb \iota(y)$ (resp.\ $\iota(x p_\lmb)$) is a quotient of $\smod_\lmb$ (resp.\ $\smod_{\rho-\lmb}$) by Remark \ref{rmk-quot-s}, and thus $xu_\lmb = 0$ (resp.\ $yu_{\rho-\lmb}=0$) implies $x p_\lmb \iota(y)=0$ (resp.\ $y \iota(x p_\lmb)=0$). 
This verifies the well-definedness of the map above. 
It is clear that the map above is a surjection. By the equality of dimensions this is in fact an isomorphism. 


Next we show $J_k / J_{k-1} \cong \smod_\lmb \otimes \smod_{\rho-\lmb}$. Since $T^{<' \overline{\lmb}}$ has a standard filtration by Proposition \ref{prop:hw-quot}, $\Ext^1(T^{<' \overline{\lmb}},T)$ vanishes. So $J_k/J_{k-1} \cong \Hom(\Ker(T^{\leq' \overline{\lmb}} \surj T^{<' \overline{\lmb}}),T)$ via the restriction map. 
The right-hand side is isomorphic to $\Hom(\smod_{\overline \lmb} \otimes (T^{\leq' \overline{\lmb}})_{\overline{\lmb}},T) \cong ((T^{\leq' \overline{\lmb}})_{\overline{\lmb}})^* \otimes \Hom(\smod_{\overline \lmb}, T)$ by the remark before the proof. 
As we have seen above, $\Hom(\smod_{\overline \lmb}, T) \cong \smod_\lmb$. 
On the other hand, since $T \cong T^* \otimes K_\rho$, $((T^{\leq' \overline{\lmb}})_{\overline{\lmb}})^* \cong (T_{\leq \mu})_\mu$ where $\mu=\rho-\overline{\lmb}$ and $T_{\leq \mu}$ denotes the largest submodule of $T$ whose weights are $\leq \mu$. 
Since $\smod_\mu \cong (P_\mu)^{\leq \mu}$ we have $(T_{\leq \mu})_\mu \cong \Hom(P_\mu, T_{\leq \mu}) \cong \Hom(\smod_\mu, T_{\leq \mu}) \cong \Hom(\smod_\mu, T) \cong \smod_{\overline{\mu}} = \smod_{\rho-\lmb}$. 

Now we show that the composition of these isomorphisms coincides with the map induced from $\phi$, up to a sign depending only on $\lmb$. Chasing the isomorphisms we see that it suffices to show $\phi(x p_\lmb y)(\tau) = \langle \iota(y)' u_{\rho-\overline{\lmb}},\tau \rangle x'u_{\overline{\lmb}}$, up to a sign depending only on $\lmb$, for all $\tau \in T_{\overline \lmb}$ and all $x, y \in A$, 
where $\langle - , - \rangle$ is a natural bilinear form on $T$ defined by $T \otimes T \overset{\text{mult.}}\mor T = \bigwedge^\bullet (K^{n-1} \oplus \cdots \oplus K^1) \surj \bigwedge^{\binom{n}{2}}(K^{n-1} \oplus \cdots \oplus K^1) \cong K$. 
Note that from the definition we see that $\langle u, x'v \rangle = \langle \iota(x)'u, v\rangle$ holds for any $u, v \in T$ and $x \in A$. 
First we have $\phi(x p_\lmb y)(\tau)=x'p_\lmb'y'\tau$. 
Here $p_\lmb' y' \tau \in T(\overline{\lmb})_{\overline{\lmb}}$ so it must be a constant multiple of $u_{\overline \lmb}$. Using the pairing defined above we see that this is equal to $\pm \langle p_\lmb' y' \tau, u_{\rho-\overline \lmb} \rangle u_{\overline \lmb}$ with the sign depending only on $\lmb$, since $u_{\overline{\lmb}} \wedge u_{\rho-\overline{\lmb}} = \pm u_{w_0}$. 
Since $\langle p_\lmb' y' \tau, u_{\rho-\overline \lmb} \rangle = \langle \tau, \iota(y)' p_{\rho-\lmb}' u_{\rho-\overline \lmb} \rangle = \langle \tau, \iota(y)' u_{\rho-\overline \lmb} \rangle$ we are done. 

\end{proof}

\begin{rmk}
By the isomorphism $\End(T) \cong A$ above we have $\End(T(\overline{\lmb})) \cong p_\lmb A p_\lmb$. But it can be seen, using $p_\lmb h =h p_\lmb$ ($h \in \csa$) and $p_\lmb e_{ij}=e_{ij} p_{\lmb-\alpha_{ij}}$, that $p_\lmb A p_\lmb \cong K$. So we see that $T(\lmb)$ is in fact an indecomposable tilting. 
\end{rmk}

\begin{rmk}
The full tilting module $T$ introduced above has a relation with double Schubert functor introduced by Sam (\cite{Sam}). 
Since the actions of $\borel$ and $\borel'$ commute, the direct sum $\borel \oplus \borel'$, which is isomorphic to the even part of the Lie superalgebra $\borel(n|n)$ (notation as in \cite{Sam}), naturally acts on $T$. 
Then it is possible to define an action of the odd part of $\borel(n|n)$ on $T$ so that $T$ is isomorphic to the double Schubert functor image $\mathscr{S}_{w_0}(V^\bullet)$ defined there.  
\end{rmk}

\section{Compatibility of Ringel duality with tensor product}
\label{sec:tensor-dual}
In this section we show that the tensor product operation and the Ringel duality functor $F = \Hom(-, T)$ ($T=\bigwedge(K^{n-1} \oplus \cdots \oplus K^1)$ as above) are in some sense compatible with each other. To be precise, we show the following: 
\begin{thm}
\label{thm:tensor-dual}
Let $M, N \in \catc_n$ have standard filtrations. 
Then $F((M \otimes N)^{\Lmb_n}) \cong (FM \otimes FN)^{\Lmb_n}$, 
where for a weight $\borel$-module $L$, $L^{\Lmb_n} \in \catc_n$ denotes the largest quotient of $L$ which is in $\catc_n$. 
\end{thm}

Let $\catc_+$ be the category of all finite dimensional weight $\borel$-modules whose weights are in $\nonneg^n$. Note that if $M, N \in \catc_+$ then $M \otimes N \in \catc_+$. 
Using the terminology from highest weight categories we say that $M \in \catc_+$ has a standard filtration if $M$ has a filtration whose successive quotients are of the form $\smod_\lmb$ ($\lmb \in \nonneg^n$). Note that, as we showed in \cite{W2}, if $M, N \in \catc_+$ have standard filtrations then $M \otimes N$ also has a standard filtration. 

\begin{rmk}
If $L \in \catc_+$ has a standard filtration, 
then as we show below, $\ch(L^{\Lmb_n}) = \ch(L)$ holds in the ring $H_n$. 
So, together with Theorem \ref{thm:dual-cn}, this theorem can be seen as a module theoretic counterpart of Proposition \ref{prop:hn-invol}; i.e. the claim that $\schub_w \mapsto \schub_{w_0ww_0}$ is a ring automorphism on $H_n$. 
\end{rmk}

First we prepare some lemmas. 

\begin{lem}
\label{lem:char-iota}
Let $\iota: H_n \mor H_n$ be the ring automorphism in Proposition \ref{prop:hn-invol}. If $M \in \catc_+$ has a standard filtration, then $\ch(FM)=\iota(\ch(M))$ in $H_n$. 
\end{lem}
\begin{proof}
Since the extensions of KP modules with $T$ vanish, if we have an exact sequence $0 \mor L \mor M \mor N \mor 0$ with $L, M, N \in \catc_+$ having standard filtrations, then $0 \mor FN \mor FM \mor FL \mor 0$ is exact. Thus we only have to show the lemma for $M=\smod_w$ ($w \in S_\infty^{(n)}$). 

The case $w \in S_n$ follows from Theorem \ref{thm:dual-cn}.
If $w \in S_\infty^{(n)} \smallsetminus S_n$ then we have $F\smod_w=\Hom(\smod_w, T)=0$ since $\smod_w$ is generated by an element of weight $\inv(w)$ while the weight space $T_{\inv(w)}$ is zero. Thus the lemma follows for this case since $\schub_w = 0$ in $H_n$. 
\end{proof}

\begin{lem}
\label{lem:char-cong}
Let $M \in \catc_+$ have a standard filtration. Then $\ch(M^{\Lmb_n}) = \ch(M)$ as elements of $H_n$.  
If $\overline{M} \in \catc_n$ is a quotient of $M$ and $\ch(\overline{M}) = \ch(M)$ in $H_n$, then $\overline{M} \cong M^{\Lmb_n}$. 
\end{lem}
\begin{proof}
By Proposition \ref{prop:hw-quot}, $\Ker(M \surj M^{\Lmb_n})$ has a filtration whose subquotients are of the form $\smod_v$ ($v \in S_\infty^{(n)} \smallsetminus S_n$). 
Thus $\ch(M) = \ch(M^{\Lmb_n}) + (\text{a linear combination of $\schub_v$ ($v \in S_\infty^{(n)} \smallsetminus S_n$)})$, and the second term vanishes in $H_n$ by Proposition \ref{prop:schub-vanish}. 
The second claim follows from the first claim since $\bigoplus_{\lmb \in \Lmb_n} \ZZ x^\lmb \cong H_n$. 
\end{proof}

\begin{lem}
\label{lem:surj-isom}
Let $M, N \in \catc_+$ have standard filtrations.
Suppose that the morphism $FM \otimes FN \mor F(M \otimes N)$ given by $\phi \otimes \psi \mapsto (m \otimes n \mapsto \phi(m) \wedge \psi(n))$ is surjective. 
Then it induces an isomorphism $(FM \otimes FN)^{\Lmb_n} \cong F(M \otimes N) \; (\cong F((M \otimes N)^{\Lmb_n}))$. 
\end{lem}
\begin{proof}
We have, as vector spaces, $F(M \otimes N) = \Hom(M \otimes N, T) = \bigoplus_{\lmb \in \Lmb_n} \Hom(M \otimes N, T(\lmb))$. It can be seen that $\Hom(M \otimes N, T(\lmb))$ is the $\overline{\lmb}$-weight space of the $\borel$-module $F(M \otimes N)$. 
Thus $F(M \otimes N) \in \catc_n$. 
By Lemma \ref{lem:char-iota} we have, in $H_n$, $\ch(F(M \otimes N)) = \iota(\ch(M)\ch(N)) = \iota(\ch(M))\iota(\ch(N)) = \ch(FM \otimes FN)$. 
Thus the claim follows from the second statement in Lemma \ref{lem:char-cong}. 
\end{proof}

For $M, N \in \catc_+$ having standard filtrations, let $\mathcal{P}(M,N)$ be the claim that the map $FM \otimes FN \mor F(M \otimes N)$ above is surjective (and thus $(FM \otimes FN)^{\Lmb_n} \cong F((M \otimes N)^{\Lmb_n})$). 
\begin{lem}
\label{lem:tensor-benri}
Let $L, M, N, X \in \catc_+$ have standard filtrations. 
Then the following implications hold: 
\begin{enumerate}
\item If $L$ is a direct sum component of $M$ then $\mathcal{P}(M, X)$ implies $\mathcal{P}(L, X)$. 
\item Suppose that there exists an exact sequence $0 \mor L \mor M \mor N \mor 0$. 
Then $\mathcal{P}(L, X) \wedge \mathcal{P}(N,X) \implies \mathcal{P}(M, X)$ and $\mathcal{P}(M, X) \implies \mathcal{P}(L, X)$ hold (in fact $\mathcal{P}(M, X)$ also implies $\mathcal{P}(N, X)$, but we do not need it here). 
\item $\mathcal{P}(L,M)$ and $\mathcal{P}(L \otimes M, N)$ implies $\mathcal{P}(L, M \otimes N)$. 
\end{enumerate}
\end{lem}
\begin{proof}
\begin{enumerate}
\item is clear since $F$ preserves direct sums. 
\item
We have a commutative diagram 
\[
\begin{CD}
0 @>>> FN \otimes FX @>>> FM \otimes FX @>>> FL \otimes FX @>>> 0 \\
@. @VVV @VVV @VVV @. \\
0 @>>> F(N \otimes X) @>>> F(M \otimes X) @>>> F(L \otimes X) @>>> 0. 
\end{CD}
\]
Here the rows are exact since $\Ext^1(N, T)$ and $\Ext^1(N \otimes X, T)$ vanish. 
This shows $\mathcal{P}(L, X) \wedge \mathcal{P}(N,X) \implies \mathcal{P}(M, X)$ and $\mathcal{P}(M, X) \implies \mathcal{P}(L, X)$. 

\item 
This holds since 
\[
\begin{CD}
FL \otimes FM \otimes FN @>>> F(L \otimes M) \otimes FN \\
@VVV @VVV \\
FL \otimes F(M \otimes N) @>>> F(L \otimes M \otimes N)
\end{CD}
\]
commutes. 
\end{enumerate}
\end{proof}

\begin{lem}
\label{lem:fm-filtr}
Let $M \in \catc_+$ have a standard filtration. Let $\lmb \in \Lmb_n$. 
Let $V \subset \Hom(M, T)$ be the submodule consisting of all homomorphisms 
which vanish on the $\mu$-weight spaces for any $\mu > \lmb$ (it is a submodule since the action of $\borel'$ on $T$ preserves weights with respect to $\csa \subset \borel$). 
Then $\Hom(M, T)/V \cong \Hom(M, T)^{<' \overline{\lmb}}$, 
the largest quotient of $\Hom(M, T)$ whose weights are all $<' \overline{\lmb}$
(recall that for $\lmb = \inv(w) \in \Lmb_n$ we defined $\overline{\lmb} = \inv(w_0 w w_0)$).  
\end{lem}
\begin{proof}
It suffices to show that the characters of both sides coincide. 

First note that $V=\Hom(M^{\not> \lmb}, T)$ where $M^{\not> \lmb}$ is the largest quotient of $M$ whose weights are all $\not> \lmb$. 
From Proposition \ref{prop:hw-quot} we see that $M^{\not> \lmb}$ has a standard filtration and, if $\ch(M)=\sum_\mu c_\mu\schub_\mu$, then the number of times $\smod_\mu$ appears in a standard filtration of $M^{\not> \lmb}$ is $c_\mu$ if $\mu \not> \lmb$ and 0 otherwise. 
Thus we see from Theorem \ref{thm:dual-cn} that $\ch(V)=\ch(\Hom(M^{\not> \lmb}, T))=\sum_{\mu \in \Lmb_n, \mu \not> \lmb} c_\mu \schub_{\overline \mu}$. 
We also see from Theorem \ref{thm:dual-cn} that $\Hom(M,T)$ has a
a standard filtration with $\smod_\mu$ appearing $c_{\overline \mu}$ times for each $\mu \in \Lmb_n$. 
Thus $\ch(\Hom(M, T)) = \sum_{\mu \in \Lmb_n} c_\mu \schub_{\overline \mu}$. 
So $\ch(\Hom(M, T)/V) = \sum_{\mu \in \Lmb_n, \mu > \lmb} c_\mu \schub_{\overline \mu}$. 

On the other hand, since $\Hom(M, T)$ has a standard filtration, by Proposition \ref{prop:hw-quot} we see 
$\ch(\Hom(M, T)^{<' \overline{\lmb}})
= \sum_{\mu \in \Lmb_n, \mu <' \overline{\lmb}} c_{\overline \mu} \schub_\mu
 = \sum_{\mu \in \Lmb_n, \overline{\mu} <' \overline{\lmb}} c_\mu \schub_{\overline \mu}
 = \sum_{\mu \in \Lmb_n, \mu > \lmb} c_\mu \schub_{\overline \mu}
 = \ch(\Hom(M, T)/V)$. This shows the claim. 
\end{proof}

Recall from the proof of Theorem \ref{thm:dual-cn} that $T$ has an action of $\borel'$, a copy of $\borel$, 
 defined by $e_{ij}' u_{pq} = \delta_{jq} u_{pi}$,
which commutes with the usual action of $\borel$. 
Recall also that we have identified $u_w$ with $\bigwedge_{(i,j) \in J(w)} u_{i,j} \in T$
where $J(w) = \{(i,\overline{j}) : i<j, w(i)>w(j)\}$. 

We write $\overline{w}=w_0ww_0$ ($w \in S_n$) and $\overline{k}=n+1-k$ ($1 \leq k \leq n$). 
For $w \in S_n$ and $p < q$ let $m_{pq}(w) = \#\{r > q : w(p)<w(r)<w(q)\}$. 
This number is precisely the number of $1 \leq r \leq n$ such that $(q,r) \in J(w)$ while $(p,r) \not\in J(w)$.
So in particular, if $(q,r) \in J(w)$ then $e_{pq}^{m_{pq}(w)} u_w = \const \cdot (u_{pr} \wedge \cdots)$ (it does not matter whether $(p,r) \in J(w)$ or not) and thus $u_{pr} \wedge e_{pq}^{m_{pq}(w)} u_w = 0$. 
Similarly, if $(r,\overline{p}) \in J(w)$ then $u_{r,\overline{q}} \wedge (e'_{\overline{q},\overline{p}})^{m_{\overline{q}, \overline{p}}(\overline{w})} u_w = 0$.

\begin{lem}
Let $w \in S_n$ and $1 \leq i \leq n-1$. For $1 \leq p, p' \leq i$ and $i+1 \leq q, q' \leq n$ such that $\ell(wt_{pq})=\ell(wt_{p'q'})=\ell(w)+1$, 
if $(e'_{\overline{q'},\overline{p'}})^{m_{\overline{q'},\overline{p'}}(\overline{w})}e_{pq}^{m_{pq}(w)} u_w \wedge u_{p,\overline{q'}} \neq 0$ then $w(p) \leq w(p')$ and $w(q) \leq w(q')$. 
Moreover,  $(e'_{\overline{q},\overline{p}})^{m_{\overline{q},\overline{p}}(\overline{w})}e_{pq}^{m_{pq}(w)} u_w \wedge u_{p,\overline{q}}$ is a nonzero multiple of $u_{wt_{pq}}$. 
\label{pqpq}
\end{lem}
\begin{proof}
First we note that the operations $e'_{\overline{q'},\overline{p'}}, e_{pq}$ and $(- \wedge u_{p,\overline{q'}})$ on $T$ all commute with each other. We have the following observations. 
\begin{enumerate}
\item If $p<p'$ and $w(p)>w(p')$ then $(p,\overline{p'}) \in J(w)$. 
Thus in this case $(e'_{\overline{q'},\overline{p'}})^{m_{\overline{q'},\overline{p'}}(\overline{w})}e_{pq}^{m_{pq}(w)} u_w \wedge u_{p,\overline{q'}} = 0$
since $(e'_{\overline{q'},\overline{p'}})^{m_{\overline{q'},\overline{p'}}(\overline{w})}u_w \wedge u_{p,\overline{q'}} = 0$. 
\item If $q<q'$ and $w(q)>w(q')$ then $(q,\overline{q'}) \in J(w)$. 
Thus in this case $(e'_{\overline{q'},\overline{p'}})^{m_{\overline{q'},\overline{p'}}(\overline{w})}e_{pq}^{m_{pq}(w)} u_w \wedge u_{p,\overline{q'}} = 0$
since $e_{pq}^{m_{pq}(w)} u_w \wedge u_{p,\overline{q'}} = 0$. 
\item If $w(p)>w(q')$ then $(p,\overline{q'}) \in J(w)$ (note that $p<q'$ holds automatically). 
Thus in this case $(e'_{\overline{q'},\overline{p'}})^{m_{\overline{q'},\overline{p'}}(\overline{w})}e_{pq}^{m_{pq}(w)} u_w \wedge u_{p,\overline{q'}} = 0$
since $u_w \wedge u_{p,\overline{q'}} = 0$. 
\end{enumerate}

Assume $(e'_{\overline{q'},\overline{p'}})^{m_{\overline{q'},\overline{p'}}(\overline{w})}e_{pq}^{m_{pq}(w)} u_w \wedge u_{p,\overline{q'}} \neq 0$. 
First we see that $w(p) < w(q')$ by (3) above. If $w(p') < w(p) < w(q')$ then by $\ell(wt_{p'q'})=\ell(w)+1$ we have $p<p'$, but then it contradicts to (1) above. Thus $w(p) \leq w(p')$. By a similar argument (using (2) instead of (1)) we see $w(q) \leq w(q')$. This shows the first claim.

It can be seen that $J(wt_{pq}) = J(w) \cup (\{p\} \times X) \smallsetminus (\{q\} \times X) \cup (Y \times \{\overline{q}\}) \smallsetminus (Y \times \{\overline{p}\}) \cup \{(p,\overline{q})\}$ where $X=\{\overline{r} : q < r , w(p)<w(r)<w(q)\}$ and $Y=\{r : r < p, w(p)<w(r)<w(q)\} = \{r : \overline{r} > \overline{p}, \overline{w}(\overline{q}) < \overline{w}(\overline{r})<\overline{w}(\overline{p})\}$. This shows the second claim.
\end{proof}

\begin{proof}[Proof of Theorem \ref{thm:tensor-dual}]
First we show that $\mathcal{P}(\smod_{w},\smod_{s_i})$ holds for any $w \in S_n$ and any $1 \leq i \leq n-1$. 

Recall that the isomorphism $\smod_{\overline{w}} \mor \Hom(\smod_w, T)$ was given by $x u_{\overline{w}} \mapsto (v \mapsto x'v)$. 
Thus we want to show that the map
$\phi: \smod_{\overline{w}} \otimes K^{n-i} \mor F(\smod_w \otimes K^i)$
given by $y u_{\overline{w}} \otimes u_q \mapsto (x u_w \otimes u_p \mapsto xy'u_w \wedge u_{pq})$
is a surjection. 

Let $(p_1, q_1), \ldots, (p_r, q_r)$ be all the pairs $(p,q)$ such that $1 \leq p \leq i < q \leq n$ and $\ell(wt_{pq})=\ell(w)+1$, 
ordered by the lexicographic order of $(w(p),w(q))$. 
Let $w^k=wt_{p_kq_k}$. 
Then $\inv(w^1) < \cdots < \inv(w^r)$ and $\inv(\overline{w^1}) >' \cdots >' \inv(\overline{w^r})$. 

For an $x \in S_n$ and $1 \leq p < q \leq n+1$ such that $\ell(xt_{pq})=\ell(x)+1$, let $v_{pq}(x) = e_{pq}^{m_{pq}(x)}u_x \otimes u_p \in \smod_x \otimes K^n$ (note that this definition is also valid for $q=n+1$ since $m_{pq}(x)=0$ in such case). 
Note that $v_{pq}(x)$ has weight $\inv(xt_{pq})$. 
As we showed in \cite[Lemma 3.3]{W2}, $\{v_{pq}(x) : 1 \leq p \leq i < q \leq n+1, \ell(xt_{pq})=\ell(x)+1\}$ generates $\smod_x \otimes K^i$ as a $\borel$-module. 

For $0 \leq k \leq r$, let $U_k$ be the submodule of $\smod_w \otimes K^i$ generated by $v_{p_l, q_l}(w)$ ($l > k$) together with $v_{j,n+1}(w)$ ($1 \leq j \leq i, \ell(wt_{j,n+1})=\ell(w)+1$). Note that $U_0=\smod_w \otimes K^i$. 
From the result of \cite[\S3]{W2} we see that $U_{k-1}/U_k \cong \smod_{w^k}$. 
In particular the weights of $U_{k-1} / U_k$ is all $\leq \inv(w^k)$, and since $\inv(w^1) \leq \cdots \leq \inv(w^k)$ we see that the weights of $(\smod_w \otimes K^i) / U_k$ are all $\leq \inv(w^k)$. 
Moreover, $U_r$ has a filtration by modules $\smod_{wt_{j, n+1}}$, 
and thus $\ch(U_r)=0$ in $H_n$. 
Therefore $(\smod_w \otimes K^i)/U_r \cong (\smod_w \otimes K^i)^{\Lmb_n}$ by Lemma \ref{lem:char-cong} (note that $(\smod_w \otimes K^i)/U_r \in \catc_n$ since $\smod_{w^1}, \ldots, \smod_{w^r} \in \catc_n$). 

Let $V_k$ ($k=1, \ldots, r$) be the submodule of $F(\smod_w \otimes K^i) = \Hom(\smod_w \otimes K^i, T)$ consisting of the homomorphisms which vanish on the $\mu$-weight spaces for any $\mu > \inv(w^k)$. By Lemma \ref{lem:fm-filtr}, $F(\smod_w \otimes K^i)/V_k \cong F(\smod_w \otimes K^i)^{<' \inv(\overline{w^k})}$ ($1 \leq k \leq r$). 
We see $V_r=F(\smod_w \otimes K^i)$ since by the argument above the weights of  $(\smod_w \otimes K^i)^{\Lmb_n}$ are all $\leq \inv(w^r)$. 
We also set $V_0=0$. 

Note that the constituents in a standard filtration of $F(\smod_w \otimes K^i)$ are
$\smod_{\overline{w^1}}, \ldots, \smod_{\overline{w^r}}$ by Theorem \ref{thm:dual-cn}. 
In particular, the only constituent $\smod_x$ with $\inv(\overline{w^{k-1}}) >' \inv(x) \geq' \inv(\overline{w^{k}})$ is $\smod_{\overline{w^k}}$. 
Thus $V_k / V_{k-1} \cong \Ker(F(\smod_w \otimes K^i)^{<' \inv(\overline{w^{k-1}})} \surj F(\smod_w \otimes K^i)^{<' \inv(\overline{w^{k}})}) \cong \smod_{\overline{w^k}}$ by Proposition \ref{prop:hw-quot}. 
In particular any nonzero element of weight $\inv(\overline{w^k})$ in $V_k/V_{k-1}$ generates $V_k / V_{k-1}$. 

We show $\phi(v_{\overline{q_k},\overline{p_k}}(\overline{w})) \in V_k \smallsetminus V_{k-1}$ for each $k$. 
Note that the desired surjectivity of $\phi$ follows from this claim since it shows that $\phi(v_{\overline{q_k},\overline{p_k}}(\overline{w}))+V_{k-1}$ is a cyclic generator of $V_k / V_{k-1}$, i.e. $\ualg(\borel)(\phi(v_{\overline{q_k},\overline{p_k}}(\overline{w}))+V_{k-1}) = V_k$. 

For $1 \leq k,l \leq r$ we have $\phi(v_{\overline{q_k},\overline{p_k}}(\overline{w}))(v_{p_l, q_l}(w)) = (e_{p_lq_l}^{m_{p_lq_l}(w)}(e_{\overline{q_k}, \overline{p_k}}')^{m_{\overline{q_k}, \overline{p_k}}(\overline{w})}u_w) \wedge u_{p_l, \overline{q_k}}$. 
By Lemma \ref{pqpq}, if $\phi(v_{\overline{q_k},\overline{p_k}}(\overline{w}))(v_{p_l, q_l}(w)) \neq 0$ then $w(p_l) \leq w(p_k)$ and $w(q_l) \leq w(q_k)$ and thus in particular $l \leq k$. 
Thus $\phi(v_{\overline{q_k},\overline{p_k}}(\overline{w}))$ induces a map $(\smod_w \otimes K^i)/U_k \mor T$ (note that the elements $v_{j, n+1}(w)$ obviously vanish under $\phi(v_{\overline{q_k}, \overline{p_k}}(\overline{w}))$ since $T$ does not have the corresponding weights). 
Since the weights of $(\smod_w \otimes K^i)/U_k$ are all $\leq \inv(w^k)$, 
this shows $\phi(v_{\overline{q_k},\overline{p_k}}(\overline{w})) \in V_k$. 
Moreover $\phi(v_{\overline{q_k},\overline{p_k}}(\overline{w}))(v_{p_k, q_k}(w)) \neq 0$ by Lemma \ref{pqpq}, and since $v_{p_k, q_k}(w)$ has weight $\inv(w^k)$ this shows $\phi(v_{\overline{q_k},\overline{p_k}}(\overline{w})) \not\in V_{k-1}$. 
Therefore we checked the claim and thus $\mathcal{P}(\smod_w, \smod_{s_i})$ follows. 

Now we can proceed to the general case. 
From (2) of Lemma \ref{lem:tensor-benri} we see that $\mathcal{P}(M, \smod_{s_i})$ holds for any $M$ having a standard filtration. 
Since if $M$ has a standard filtration then $M \otimes \smod_{s_i}$ also has a standard filtration, 
(3) of Lemma \ref{lem:tensor-benri} shows that $\mathcal{P}(M, \smod_{s_i} \otimes \smod_{s_j} \otimes \cdots)$ holds for any $i, j, \ldots$ and any $M$. 
Then from (1) of Lemma \ref{lem:tensor-benri} we see that $\mathcal{P}(M, T(\lmb))$ holds for any $\lmb$ and any $M$, since $T(\lmb)$ is a direct sum component of $\bigotimes_{1 \leq i \leq n-1} (\smod_{s_{n-i}})^{\otimes \overline{\lmb}_i}$.
Thus again from (2) of Lemma \ref{lem:tensor-benri} we get $\mathcal{P}(M, \smod_{\lmb})$, 
since as we showed in \cite{W2} there is an injection $\smod_\lmb \inj T(\lmb)$ such that its cokernel admits a standard filtration. Thus $\mathcal{P}(M, N)$ for general $M, N$ follows by (2) of Lemma \ref{lem:tensor-benri}. 
\end{proof}

\begin{rmk}
As we saw, $(M, N) \mapsto (M \otimes N)^{\Lmb_n}$ is a very fundamental operation in the category $\catc_n^\std$; this in fact defines a structure of symmetric tensor category on $\catc_n^\std$. 
Experimental results suggest an interesting conjecture relating this ``restricted tensor product'' operation and our full tilting module $T$: the dimension of $(T^{\otimes k})^{\Lmb_n}$ seems to be $(k+1)^{\binom{n}{2}}$ for any $k$. Also there is a finer form of this conjecture: the dimension of the degree-$d$ piece (with respect to the grading induced from the natural grading on $T=\bigwedge^\bullet(\cdots)$) of $(T^{\otimes k})^{\Lmb_n}$ seems to be $k^d\binom{\binom{n}{2}}{d}$. 

It can be shown that the latter version of the conjecture implies that $\Hom(T^{\otimes k}, T)$ also has a dimension $(k+1)^{\binom{n}{2}}$. Note that this is true for $k=1$ since $\ch(\End(T)) = \iota(\ch(T))=\sum_{v \in S_n} \schub_v(x) \schub_{w_0v}(1)$ and thus $\dim(\End(T))=\sum_{v \in S_n} \schub_v(1) \schub_{w_0v}(1) = \sum_{v \in S_n} \schub_{v^{-1}}(x) \schub_{(w_0v)^{-1}}(1) = \sum_{v \in S_n} \schub_v(1) \schub_{vw_0}(1) = 2^{\binom{n}{2}}$ by Cauchy formula. 
\end{rmk}

\appendix

\section{Appendix: highest weight categories}
In this appendix we summarize and give proofs for results on highest weight categories used in this paper. Some of them appear in references such as \cite{CPS}, \cite{Rin} and \cite[Appendix]{Don}, but we also give proofs for them here to adapt to our settings because the formulations of highest weight categories and these properties vary with references. Our treatment of tilting objects and Ringel duality mostly follows \cite[Appendix]{Don}, with some minor changes and improvements on the arguments. 

In the following let $\catc$ denote a highest weight category (Definition \ref{defn-hwc}), $\Lmb$ be its weight poset, and $L(\lmb), P(\lmb), Q(\lmb)$ and $\std(\lmb)$ stand for the simple, projective, injective and standard objects respectively. Also, let $\costd(\lmb)$ denote the costandard objects, i.e. $\costd(\lmb)$ is the injective hull of $L(\lmb)$ in $\catc_{\leq \lmb}$. We denote the head and socle of an object $M \in \catc$ by $\hd M$ and $\soc M$ respectively. 

For an order ideal $\Lmb' \subset \Lmb$ we denote by $\catc_{\Lmb'}$ the full subcategory of $\catc$ consisting of the objects such that its simple constituents are $L(\lmb)$ ($\lmb \in \Lmb'$). We denote $\catc_{\leq \lmb}$ etc. to mean $\catc_{\{\mu : \mu \leq \lmb\}}$ etc. For $M \in \catc$ let $M^{\Lmb'}$ be the largest quotient of $M$ which is in $\catc_{\Lmb'}$, and write $M^{\leq \lmb}$ etc. to mean $M^{\{\mu : \mu \leq \lmb\}}$ etc. 

\begin{rmk}
$\Ker(M \surj M^{\Lmb'})$ does not have any $L(\lmb)$ ($\lmb \in \Lmb'$) as its quotient: if $\Ker(M \surj M^{\Lmb'}) / N \cong L(\lmb)$ is such a quotient, then $M/N$ would be a quotient of $M$, its simple constituents are $L(\nu)$ ($\nu \in \Lmb'$), and it is strictly larger than $M^{\Lmb'}$: these contradict to the definition of $M^{\Lmb'}$. 
\label{quot-rem}
\end{rmk}

For an $M \in \catc$ and $\lmb \in \Lmb$ let $(M : L(\lmb))$ denote the number of times $L(\lmb)$ appears in the simple constituents of $M$. It can be easily seen that $\dim \Hom(P(\lmb), M) = (M : L(\lmb)) \dim \Hom(P(\lmb), L(\lmb))$.

\subsection{Basic Facts}
\begin{prop}
There is a surjection $\std(\lmb) \surj L(\lmb)$ such that the simple constituents of the kernel are of the form $L(\mu)$ ($\mu < \lmb$). 
\end{prop}
\begin{proof}
First we show that $(\std(\lmb) : L(\mu)) \neq 0$ implies $\mu \leq \lmb$. Assume $(\std(\lmb) : L(\mu)) \neq 0$. 
This means $\Hom(P(\mu), \std(\lmb)) \neq 0$. Since $P(\mu)$ has a filtration by $\std(\nu)$ ($\nu \geq \mu$) it follows that $\Hom(\std(\nu), \std(\lmb)) \neq 0$ for some $\nu \geq \mu$. Thus $\mu \leq \nu \leq \lmb$. 

Next we see $(\std(\lmb) : L(\lmb)) = 1$. 
Since $\Ker(P(\lmb) \surj \std(\lmb))$ has a filtration by $\std(\nu)$ ($\nu > \lmb$)
we see that $\Hom(\Ker(P(\lmb) \surj \std(\lmb)), \std(\lmb)) = 0$. 
Thus we have an exact sequence $0 \mor \Hom(\std(\lmb), \std(\lmb)) \mor \Hom(P(\lmb), \std(\lmb)) \mor \Hom(\Ker(P(\lmb) \surj \std(\lmb)), \std(\lmb)) = 0$ and thus $\Hom(P(\lmb), \std(\lmb)) \cong \End(\std(\lmb)) \cong K$. 
Therefore $(\std(\lmb) : L(\lmb)) = 1$. 

Finally we show that $L(\lmb)$ is a quotient of $\std(\lmb)$. Since $(\std(\lmb) : L(\lmb)) = 1$, there exists an $N \subset \std(\lmb)$ and a surjection $f : N \surj L(\lmb)$. By the projectivity of $P(\lmb)$, the surjection $\pi : P(\lmb) \surj L(\lmb)$ factors as $\pi = fg$ for some $g : P(\lmb) \mor N$. The composition $P(\lmb) \mor N \inj \std(\lmb)$ is nonzero and thus must be a nonzero multiple of the surjection $P(\lmb) \surj \std(\lmb)$ since as we saw above $\Hom(P(\lmb), \std(\lmb)) \cong K$. But the image of the composition map above is $N$, so we get $N = \std(\lmb)$. Thus the claim follows. 
\end{proof}

By the proposition above $\std(\lmb) \in \catc_{\Lmb'}$ for any order ideal $\Lmb'$ containing $\lmb$. Also from the proof we see $\Hom(P(\lmb), L(\lmb)) \cong K$. 

\begin{prop}
$\Hom(\std(\lmb), L(\mu)) = 0$ for $\mu \neq \lmb$ and $\Hom(\std(\lmb), L(\lmb)) \cong K$. Thus in particular $\hd \std(\lmb) \cong L(\lmb)$. 
\label{hom-std-simp}
\end{prop}
\begin{proof}
This can be easily seen from the exact sequence $0 \mor \Hom(\std(\lmb), L(\mu)) \mor \Hom(P(\lmb), L(\mu))$ since the last term is $0$ for $\mu \neq \lmb$ and $K$ for $\mu = \lmb$. 
\end{proof}

\subsection{Projectivities of Standard objects}
\begin{prop}
$\Ext^1(\std(\lmb), L(\mu)) \neq 0$ implies $\lmb < \mu$. So $\std(\lmb)$ is projective in $\catc_{\Lmb'}$ for any order ideal $\Lmb'$ which contains $\lmb$ as a maximal element. 
\label{ext-std-simp}
\end{prop}
Because the simple constituents of $\std(\mu)$ are $L(\nu)$ ($\nu \leq \mu$) we get as a corollary: 
\begin{cor}
$\Ext^1(\std(\lmb), \std(\mu)) \neq 0$ implies $\lmb < \mu$. \hfill $\Box$
\label{ext-std-std}
\end{cor}
\begin{proof}[Proof of the Proposition \ref{ext-std-simp}]
Assume $\Ext^1(\std(\lmb), L(\mu)) \neq 0$. Let $M = \Ker(P(\lmb) \surj \std(\lmb))$, so $M$ has a filtration by $\std(\nu)$ ($\nu > \lmb$). 
By the exact sequence $\Hom(M, L(\mu)) \mor \Ext^1(\std(\lmb), L(\mu)) \mor \Ext^1(P(\lmb), L(\mu))=0$ we see that $\Hom(M, L(\mu)) \neq 0$. This implies that $\Hom(\std(\nu), L(\mu)) \neq 0$ for some $\nu > \lmb$. So $\mu = \nu > \lmb$ by Proposition \ref{hom-std-simp}. 
\end{proof}

Since $\hd \std(\lmb) \cong L(\lmb)$ by Proposition \ref{hom-std-simp} we get:
\begin{prop}
Let $\Lmb' \subset \Lmb$ be an order ideal and let $\lmb \in \Lmb'$ be a maximal element. Then $\std(\lmb) \surj L(\lmb)$ is a projective cover in $\catc_{\Lmb'}$ (so $\std(\lmb) \cong P(\lmb)^{\Lmb'}$). \hfill $\Box$
\label{std-projectivity}
\end{prop}

\subsection{Hom and Ext between Standard and Costandard Objects}

\begin{prop}
$\Ext^i(\std(\lmb), \costd(\mu)) \cong K$ iff $\lmb = \mu$ and $i=0$, and otherwise $0$. 
\label{std-costd-ext}
\end{prop}
\begin{proof}
We have an exact sequence $0 \mor \Hom(L(\lmb),\costd(\lmb)) \mor \Hom(\std(\lmb), \costd(\lmb)) \mor \Hom(\Ker(\std(\lmb) \surj L(\lmb)), \costd(\lmb))$. 
Here the simple constituents of the kernel are $L(\nu)$ ($\nu < \lmb$), and $\Hom(L(\nu), \costd(\lmb)) = 0$ for $\nu < \lmb$ since $\soc \costd(\lmb) \cong L(\lmb)$. Thus the last term of the sequence above vanishes. 
Also, $\Hom(L(\lmb), \costd(\lmb)) \cong \End(L(\lmb)) \cong \Hom(P(\lmb), L(\lmb)) \cong K$ since $\soc \costd(\lmb) \cong L(\lmb)$ and $\hd P(\lmb) \cong L(\lmb)$. Thus $\Hom(\std(\lmb), \costd(\lmb)) \cong K$. 

We show the vanishings of the other extensions. 
\begin{itemize}
\item $i=0$: $\Hom(\std(\lmb), \costd(\mu)) \neq 0$ implies that $\Hom(L(\nu), \costd(\mu)) \neq 0$ for some $\nu \leq \lmb$. But since $\soc \costd(\mu) \cong L(\mu)$ this means that $\mu = \nu \leq \lmb$. Thus $\Hom(\std(\lmb), \costd(\mu)) \neq 0$ implies $\mu \leq \lmb$. 
By the same argument (using $\hd \std(\lmb) \cong L(\lmb)$ instead of $\soc \costd(\mu) \cong L(\mu)$) we see that $\Hom(\std(\lmb), \costd(\mu)) \neq 0$ also implies $\mu \geq \lmb$. Thus $\Hom(\std(\lmb), \costd(\mu)) \neq 0$ implies $\lmb=\mu$. 

\item $i=1$: Note that $\Ext^1 = \Ext^1_{\catc_{\Lmb'}}$ for any $\Lmb'$ since $\catc_{\Lmb'}$ is closed under extensions. If $\lmb \leq \mu$ then $\std(\lmb) \in \catc_{\leq \mu}$
and thus $\Ext^1(\std(\lmb), \costd(\mu)) = \Ext^1_{\catc_{\leq \mu}}(\std(\lmb),\costd(\mu)) = 0$ by the injectivity of $\costd(\mu) \in \catc_{\leq \mu}$. 
Otherwise $\costd(\mu) \in \catc_{\not> \lmb}$
and thus $\Ext^1(\std(\lmb), \costd(\mu)) = \Ext^1_{\catc_{\not> \lmb}}(\std(\lmb), \costd(\mu)) = 0$ by the projectivity of $\std(\lmb) \in \catc_{\not> \lmb}$. 
\item $i \geq 2$ : Follows from the exact sequence $0 = \Ext^{i-1}(P(\lmb), \costd(\mu)) \mor \Ext^{i-1}(\Ker(P(\lmb) \surj \std(\lmb)), \costd(\mu)) \mor \Ext^i(\std(\lmb),\costd(\mu)) \mor \Ext^i(P(\lmb), \costd(\mu)) = 0$ and the downward induction on $\lmb$. 
\end{itemize}
\end{proof}

Since $\Hom(P(\lmb), \costd(\lmb)) = \Hom(P(\lmb)^{\leq \lmb}, \costd(\lmb)) = \Hom(\std(\lmb), \costd(\lmb)) \cong K$ we see that $(\costd(\lmb) : L(\lmb)) = 1$. 

\subsection{Standard Filtration}

A \emph{standard (resp.\ costandard) filtration} of an object is a filtration such that each of its successive quotients are standard (resp.\ costandard) objects. Let $\catc^\std$ denote the full subcategory of the objects having standard filtrations. 

\begin{prop}
For $M \in \catc$ having a standard (resp.\ costandard) filtration, the number of times $\std(\lmb)$ (resp.\ $\costd(\lmb)$) appears in (any) standard (resp.\ costandard) filtration of $M$ is given by $\dim \Hom(M, \costd(\lmb))$ (resp.\ $\dim \Hom(\std(\lmb), M)$). 
\label{filtr-times}
\end{prop}
\begin{proof}
This is immediate from Proposition \ref{std-costd-ext}. 
\end{proof}

\begin{prop}
Let $\Lmb' \subset \Lmb$ be an order ideal and let $\lmb \in \Lmb'$ be a maximal element. 
Then for $M \in \catc^\std$, $\Ker(M^{\Lmb'} \surj M^{\Lmb' \smallsetminus \{\lmb\}})$ is a direct sum of copies of $\std(\lmb)$. 
\label{filtr-lem}
\end{prop}
Note that the proposition in particular implies that $M^{\Lmb'}$ and $\Ker(M \surj M^{\Lmb'})$ are in $\catc^\std$ for any order ideal $\Lmb' \subset \Lmb$. 
\begin{proof}
First note that the head of $\Ker(M^{\Lmb'} \surj M^{\Lmb' \smallsetminus \{\lmb\}})$ is a direct sum of copies of $L(\lmb)$ by Remark \ref{quot-rem}. 
Thus the projective cover, in $\catc_{\Lmb'}$, of this head is a direct sum of some copies of $\std(\lmb)$. So it suffices to show that $\Ker(M^{\Lmb'} \surj M^{\Lmb' \smallsetminus \{\lmb\}})$ is also a projective cover of $\hd \Ker(M^{\Lmb'} \surj M^{\Lmb' \smallsetminus \{\lmb\}})$, i.e. $\Ker(M^{\Lmb'} \surj M^{\Lmb' \smallsetminus \{\lmb\}})$ is projective in $\catc_{\Lmb'}$. 

Let $M' = M^{\Lmb'}$, $M''=M^{\Lmb' \smallsetminus \{\lmb\}}$ and $N = \Ker(M' \surj M'')$. 
We want to show that $\Ext^1(N, L(\mu))$ vanish for all $\mu \in \Lmb'$. 
We have exact sequences $\Hom(N, \costd(\mu)/L(\mu)) \mor \Ext^1(N, L(\mu)) \mor \Ext^1(N, \costd(\mu))$, $\Ext^1(M', \costd(\mu)) \mor \Ext^1(N, \costd(\mu)) \mor \Ext^2(M'', \costd(\mu))$ and $\Hom(\Ker(M \surj M'), \costd(\mu)) \mor \Ext^1(M', \costd(\mu)) \mor \Ext^1(M, \costd(\mu))$. Here
\begin{itemize}
\item $\Ext^1(M, \costd(\mu))$ vanishes by Proposition \ref{std-costd-ext}. 
\item $\Ext^2(M'', \costd(\mu))$ vanishes by Proposition \ref{std-costd-ext}, since $M''$ has a standard filtration by induction on $|\Lmb'|$. 
\item $\Hom(N, \costd(\mu)/L(\mu))$ and $\Hom(\Ker(M \surj M'), \costd(\mu))$ vanishes by Remark \ref{quot-rem} since the simple constituents of $\costd(\mu)/L(\mu)$ (resp.\ $\costd(\mu)$) are $L(\nu)$ ($\nu < \mu$ (resp.\ $\nu \leq \mu$)). 
\end{itemize}
And thus $\Ext^1(N, L(\mu)) = 0$ as desired. 
\end{proof}

Also from the above proof we get the following corollary:
\begin{cor}
$M \in \catc$ has a standard filtration if and only if $\Ext^1(M, \costd(\lmb)) = 0$ for all $\lmb \in \Lmb$. \hfill $\Box$
\label{filtr-crit}
\end{cor}

By Proposition \ref{std-costd-ext} and Corollary \ref{filtr-crit} we get the following: 
\begin{cor}
if $0 \mor L \mor M \mor N \mor 0$ is an exact sequence in $\catc$ and $M, N \in \catc^\std$, then $L \in \catc^\std$. 
\label{filtr-ker}
\end{cor}
\begin{proof}
By Proposition \ref{std-costd-ext} we have $\Ext^1(M, \costd(\lmb)) = 0$ and $\Ext^2(N, \costd(\lmb))=0$ for any $\lmb \in \Lmb$. Thus by the exact sequence $\Ext^1(M, \costd(\lmb)) \mor \Ext^1(L, \costd(\lmb)) \mor \Ext^2(N, \costd(\lmb))$ we see $\Ext^1(L, \costd(\lmb))=0$. Thus by Corollary \ref{filtr-crit} we see $L \in \catc^\std$.  
\end{proof}

For an $M$ having a standard filtration let $(M : \std(\lmb))$ denote the number of times $\std(\lmb)$ appears in a standard filtration of $M$ (which does not depend on a choice of filtration by Corollary \ref{filtr-times}).

\subsection{Tilting Objects}
\begin{defn}
An object $T \in \catc$ is called a \emph{tilting} or a \emph{tilting object} if it has a standard filtration and $\Ext^1(\std(\lmb), T) = 0$ for all $\lmb \in \Lmb$. 
\end{defn}
Note that if $T$ is a tilting then so are its direct summands, 
because $T \in \catc$ is a tilting if and only if $\Ext^1(T, \costd(\lmb))$ and $\Ext^1(\std(\lmb), T)$ vanish for all $\lmb$. 

For $M \in \catc^\std$, define $\supp M \subset \Lmb$ as the order ideal of $\Lmb$ generated by all $\lmb$ such that $(M : \std(\lmb)) \neq 0$. 
For an $X \subset \Lmb$ let $X^\circ$ be the set of all non-maximal elements in $X$. 

\begin{prop}
Let $M \in \catc^\std$.
Then there is a tilting $T$ and an injection $M \inj T$ such that $\supp T = \supp M$, $T/M \in \catc^\std$ and $\supp(T/M) \subset (\supp M)^\circ$. 
\label{tilt-embed}
\end{prop}
\begin{proof}
For an $M \in \catc^\std$, define $\defect M \subset \Lmb$, the \emph{defect} of $M$, to be the order ideal of $\Lmb$ generated by $\{ \lmb \in \Lmb : \Ext^1(\std(\lmb), M) \neq 0\}$. Note that $M \in \catc^\std$ is a tilting if and only if $\defect M = \varnothing$. 

If $\defect M = \varnothing$ then we are done. Assume $\defect M \neq \varnothing$. We embed $M$ into an $\tilde{M} \in \catc^\std$ with strictly smaller defect.

Take a maximal element $\lmb \in \defect M$. Then $\Ext^1(\std(\lmb), M) \neq 0$, and thus there exists a nonsplit exact sequence $0 \mor M \mor M_1 \mor \std(\lmb) \mor 0$. 
Assume $\Ext^1(\std(\mu), M_1) \neq 0$ for some $\mu \in \Lmb$. Since there is an exact sequence $\Ext^1(\std(\mu), M) \mor \Ext^1(\std(\mu), M_1) \mor \Ext^1(\std(\mu), \std(\lmb))$ it follows that either $\Ext^1(\std(\mu), M) \neq 0$ or $\Ext^1(\std(\mu), \std(\lmb)) \neq 0$. The first one implies $\mu \in \defect M$, while the second one implies $\mu < \lmb$ by Corollary \ref{ext-std-std}. The latter case implies $\mu \in \defect M$ and thus $\mu \in \defect M$ in either case. This shows $\defect M_1 \subset \defect M$. 
Moreover we claim that $\dim \Ext^1(\std(\lmb), M_1) < \dim \Ext^1(\std(\lmb), M)$. In fact, we have an exact sequence $\Hom(\std(\lmb), M_1) \mor \Hom(\std(\lmb), \std(\lmb)) \mor \Ext^1(\std(\lmb), M) \mor \Ext^1(\std(\lmb), M_1) \mor \Ext^1(\std(\lmb), \std(\lmb))$ where the last term is zero by Corollary \ref{ext-std-std}. But here $\Hom(\std(\lmb), M_1) \mor \Hom(\std(\lmb), \std(\lmb))$ is a zero map: otherwise we would have a morphism $\std(\lmb) \mor M_1$ such that the composition $\std(\lmb) \mor M_1 \surj \std(\lmb)$ is nonzero and thus an isomorphism (since $\End(\std(\lmb)) \cong K$), which contradicts to the assumption that $M_1 \surj \std(\lmb)$ is nonsplit. Thus we have an exact sequence $0 \mor \End(\std(\lmb)) \mor \Ext^1(\std(\lmb), M) \mor \Ext^1(\std(\lmb), M_1) \mor 0$ and this shows the claim. 

Repeating the construction above we have an $\tilde{M} \in \catc^\std$ and $M \inj \tilde{M}$ such that $\defect \tilde{M} \subset \defect M \smallsetminus \{\lmb\}$. 
Repeating again we get an embedding $M \inj T$ into a tilting. It is clear from the construction that $T/M \in \catc^\std$. 

We claim that $\supp(T/M) \subset (\supp M)^\circ$. By the construction it suffices to show that $\defect M \subset (\supp M)^\circ$. 
Assume $\Ext^1(\std(\lmb), M) \neq 0$ for some $\lmb$. Then $\Ext^1(\std(\lmb), \std(\mu)) \neq 0$ for some $\mu \in \supp M$. Since $\lmb < \mu$ by Corollary \ref{ext-std-std}, this shows that $\lmb$ is a non-maximal element in $\supp M$. This shows the claim. 
\end{proof}

By the proposition there is an embedding $\std(\lmb) \inj T$ such that $T$ is a tilting, $\supp T = \{\mu : \mu \leq \lmb\}$ and $(T : \std(\lmb)) = 1$. So there is an indecomposable summand $T(\lmb)$ of $T$ such that $(T(\lmb) : \std(\lmb)) = 1$.  By Proposition \ref{filtr-lem} we see that there in fact is an embedding $\std(\lmb) \inj T(\lmb)$ such that $T(\lmb) / \std(\lmb)$ has a standard filtration. 

Note that $\lmb$ can be recovered from $T(\lmb)$ as the unique maximal element in $\supp T(\lmb)$: in particular $T(\lmb) \not\cong T(\mu)$ if $\lmb \neq \mu$. 

\begin{prop}
Every tilting is a direct sum of the objects $T(\lmb)$. 
\label{indec-tilt}
\end{prop}
\begin{proof}
Let $T\neq 0$ be a tilting. Take a maximal element $\lmb \in \supp T$. 
We show that there is a split surjection $T \surj T(\lmb)$: this inductively shows the claim. 

By the maximality of $\lmb$ we see $(T : \std(\lmb)) \neq 0$. This implies, by Proposition \ref{filtr-lem} and the maximality of $\lmb$, that there is an injection $\std(\lmb) \inj T$ with cokernel $T / \std(\lmb)$ having a standard filtration. 
We name the morphisms $\std(\lmb) \inj T(\lmb)$ and $\std(\lmb) \inj T$ as $f$ and $g$ respectively. 

We have exact sequences $\Hom(T, T(\lmb)) \mor \Hom(\std(\lmb), T(\lmb)) \mor \Ext^1(T/\std(\lmb), T(\lmb)) = 0$ and $\Hom(T(\lmb), T) \mor \Hom(\std(\lmb),T) \mor \Ext^1( T(\lmb)/\std(\lmb), T) = 0$. Thus there are morphisms $h : T \mor T(\lmb)$ and $k : T(\lmb) \mor T$ such that $f=hg$ and $g=kf$. Then $f=(hk)^n f$ for any $n \geq 0$, and thus $hk \in \End(T(\lmb))$ is not nilpotent. Then by Fitting's lemma $hk$ is an isomorphism. Thus $h$ is a split surjection, as desired. 
\end{proof}

Also, repeated use of Proposition \ref{tilt-embed} shows the following:  

\begin{prop}
Any $M \in \catc^\std$ has a finite resolution $0 \mor M \mor T_0 \mor \cdots \mor T_r \mor 0$ by tiltings.  \hfill $\Box$
\label{tilt-resol}
\end{prop}

\subsection{Ringel Duality}
Let us fix a tilting object $T$ such that every indecomposable tilting occurs at least once as its direct summand (such an object is called a \emph{full tilting}). 
Let $\catc^\vee$ be the category of all finite-dimensional left $\End(T)$-modules. 
Let $F = \Hom(-,T) : \catc \mor (\catc^\vee)^\mathrm{op}$. 

Note that, since $\Ext^1(N,T)=0$ for $N \in \catc^\std$, the functor $F$ is exact on $\catc^\std$, that is, it maps an exact sequence $0 \mor L \mor M \mor N \mor 0$ with $L, M, N \in \catc^\std$ to an exact sequence $0 \mor FN \mor FM \mor FL \mor 0$. 
This observation implies a more general consequence: suppose that there is an exact sequence $\cdots \mor M_1 \mor M_0 \mor 0$ in $\catc^\std$ bounded from right. Then Corollary \ref{filtr-ker} implies that $\Ker(M_1 \mor M_0) \in \catc^\std$ and thus we can work inductively to see that $0 \mor FM_0 \mor FM_1 \mor \cdots$ is exact. 

\begin{lem}
For any $M \in \catc$ and any tilting $T'$, the map $\Hom_\catc(M, T') \mor \Hom_{\catc^\vee}(FT', FM)$ induced from $F$ is an isomorphism. 
\label{hommt-lemma}
\end{lem}
\begin{proof}
For $T' = T$ it is clear. For a general case, it can be seen from the fact that $T'$ appears as a direct summand of some $T^{\oplus m}$ ($m \gg 0$). 
\end{proof}

\begin{prop}
The indecomposable projectives in $\catc^\vee$ are given by $F T(\lmb)$ ($\lmb \in \Lmb$). 
\end{prop}
\begin{proof}
Since $\End(T)$ is, as a left $\End(T)$-module, a direct sum of the modules of the form $F T(\lmb)$ ($\lmb \in \Lmb$), it suffices to show that they are indeed indecomposable. 
By the previous lemma $\End(F T(\lmb)) \cong \End(T(\lmb))$, and since $T(\lmb)$ is indecomposable $\End(T(\lmb))$ contains no idempotents. Thus $F T(\lmb)$ is indecomposable. 
\end{proof}

\begin{prop}
For $M, N \in \catc^\std$ and any $i \geq 0$, $\Ext^i(M, N) \cong \Ext^i(FN, FM)$. For $i=0$ this isomorphism is equal to the map induced from $F$, and for $i=1$ this isomorphism is equal to the map $[0 \mor N \mor X \mor M \mor 0] \mapsto [0 \mor FM \mor FX \mor FN \mor 0]$ where these exact sequences are seen as elements in certain $\Ext^1$ groups. 
\label{ringel-ext}
\end{prop}
\begin{proof}
Take a finite tilting resolution $0 \mor N \mor T_0 \mor \cdots \mor T_r \mor 0$ of $N$ which exists by Proposition \ref{tilt-resol}. 
Then $0 \mor FT_r \mor \cdots \mor FT_0 \mor FN \mor 0$ is a projective resolution since $F$ is exact on $\catc^\std$ and $FT_i$ are projective. 
By the same argument as in \cite[Theorem 2.7.6]{Wei} we see (since $\Hom(-, T_i)$ are exact on $\catc^\std$)  
that $\Ext^i(M, N)$ is the $i$-th cohomology of the complex $\Hom(M, T_\bullet)$. 
On the other hand, $\Ext^i(FN, FM)$ is the $i$-th cohomology of $\Hom(FT_\bullet, FM)$. 
By Lemma \ref{hommt-lemma} the map induced from $F$ gives an isomorphism between these two complexes and thus the first claim follows. The latter claim for $i=0$ also follows from this argument. 

Recall the correspondence from extensions to Ext group (\cite[\S A.5]{ASS}): for a projective resolution $\cdots \mor P_1 \mor P_0 \mor M \mor 0$, there always exist $f : P_1 \mor N$ and $g : P_0 \mor X$ such that 
\[
\begin{CD}
 @. P_1 @>>> P_0 @>>> M @>>> 0 \\
@. @VVfV @VVgV @| @. \\
0 @>>> N @>>> X @>>> M @>>> 0
\end{CD}
\]
commutes, and then the element $[0 \mor N \mor X \mor M \mor 0] \in \Ext^1(M, N)$ is given by taking the class of $f \in \Hom(P_1, N)$. Chasing the double-complex argument above we see that the correspondence can also be obtained by taking $h : X \mor T_0$ and $k : M \mor T_1$ such that
\[
\begin{CD}
0 @>>> N @>>> X @>>> M @>>> 0 \\
@. @| @VVhV @VVkV @. \\
0 @>>> N @>>> T_0 @>>> T_1
\end{CD}
\]
commute and taking the class of $k \in \Hom(M, T_1)$. Applying $F$ to the diagram above we get
\[
\begin{CD}
0 @<<< FN @<<< FX @<<< FM @<<< 0 \\
@. @| @AA{Fh}A @AA{Fk}A @. \\
0 @<<< FN @<<< FT_0 @<<< FT_1
\end{CD}
\]
with rows exact and $FT_0, FT_1$ projective. Thus $[0 \mor FM \mor FX \mor FN \mor 0] \in \Ext^1(FN, FM)$ is equal to the class of $Fk \in \Hom(FT_1, FM)$ and this shows the claim. 
\end{proof}

\begin{prop}
$\catc^\vee$ is a highest weight category with weight poset $\Lmb^\mathrm{op}$, the opposite poset of $\Lmb$, and standard objects $\{F\std(\lmb)\}$. 
\label{ringel-hwc}
\end{prop}
\begin{proof}
Since $\Hom(F\std(\lmb), F\std(\mu)) \cong \Hom(\std(\mu),\std(\lmb))$ the first two axioms are clear. 

We have an exact sequence $0 \mor \std(\lmb) \mor T(\lmb) \mor M \mor 0$ such that $M$ has a filtration by $\std(\mu)$ ($\mu < \lmb$). Applying $F$ we get an exact sequence $0 \mor FM \mor FT(\lmb) \mor F\std(\lmb) \mor 0$ with $FM$ having a filtration by $F\std(\mu)$ ($\mu < \lmb$). This checks the last axiom. 
\end{proof}

\begin{prop}
$F$ restricts to a contravariant equivalence between $\catc^\std$ and $(\catc^\vee)^\std$. 
\label{ringel-dual}
\end{prop}
\begin{proof}
We saw that $F|_{\catc^\std}$ is fully faithful and thus it suffices to show the essential-surjectivity: i.e. we want to show that for any $N \in (\catc^\vee)^\std$ there exists an $M \in \catc^\std$ such that $FM \cong N$. This follows by the induction on the length of $N \in (\catc^\vee)^\std$: if $0 \mor N' \mor N \mor N'' \mor 0$ is an exact sequence with $N' \cong FM'$ and $N'' \cong FM''$ ($M', M'' \in \catc^\std$), then since $\Ext^1(N'', N') \cong \Ext^1(M', M'')$ there is an exact sequence $0 \mor M'' \mor M \mor M' \mor 0$ mapped to the above sequence under $F$, and in particular $N \cong FM$. 
\end{proof}

\end{document}